\documentclass{amsart}
\usepackage{cite}
\usepackage[T1]{fontenc}
\usepackage{indentfirst}
\usepackage{amssymb}
\usepackage{amsmath}
\usepackage{graphics}
\usepackage{epsfig}
\usepackage{xcolor}

\usepackage[utf8]{inputenc}
\usepackage{latexsym}
\usepackage{graphicx}

\usepackage[normalem]{ulem}
\usepackage{cancel,soul}

\usepackage[colorlinks=true, linkcolor=purple, urlcolor=marino,citecolor=purple]{hyperref}

\catcode`@=11 \@addtoreset{equation}{section}
\renewcommand\theequation{\thesection.\@arabic\c@equation}
\catcode`@=12

\expandafter\chardef\csname pre amssym.def
at\endcsname=\the\catcode`\@ \catcode`\@=11
\def\undefine#1{\let#1\undefined}
\def\newsymbol#1#2#3#4#5{\let\next@\relax
	\ifnum#2=\@ne\let\next@\msafam@\else
	\ifnum#2=\tw@\let\next@\msbfam@\fi\fi
	\mathchardef#1="#3\next@#4#5}
\def\mathhexbox@#1#2#3{\relax
	\ifmmode\mathpalette{}{\m@th\mathchar"#1#2#3}%
	\else\leavevmode\hbox{$\m@th\mathchar"#1#2#3$}\fi}
\def\hexnumber@#1{\ifcase#1 0\or 1\or 2\or 3\or 4\or 5\or 6\or 7\or 8\or
	9\or A\or B\or C\or D\or E\or F\fi}

\font\teneufm=eufm10 \font\seveneufm=eufm7 \font\fiveeufm=eufm5
\newfam\eufmfam
\textfont\eufmfam=\teneufm \scriptfont\eufmfam=\seveneufm
\scriptscriptfont\eufmfam=\fiveeufm

\catcode`\@=\csname pre amssym.def at\endcsname

\newtheorem{theorem}{Theorem}[section]
\newtheorem{lemma}[theorem]{Lemma}

\newtheorem{corollary}[theorem]{Corollary}
\newtheorem{definition}[theorem]{Definition}

\newtheorem{remark}[theorem]{Remark}
\renewcommand{\theequation}{\thesection.\arabic{equation}}
\begin{document}
	
	\title[optimal control for weak solutions of the Keller-Segel]{Bilinear optimal control for weak solutions of the Keller-Segel logistic  model  in $2D$ domains}
	
	\author[\,\,\,\,\,\,\,\,\,P. Braz e Silva]{P. Braz e Silva}
	\address[P. Braz e Silva]{Departamento de Matem\'atica. Universidade Federal de Pernambuco, 
		CEP 50740-560, Recife - PE. Brazil}
	\email{pablo.braz@ufpe.br}
	
	\author[F. Guill\'en-Gonz\'alez]{F. Guill\'en-Gonz\'alez}
	\address[F. Guill\'en-Gonz\'alez]{Dpto. de Ecuaciones Diferenciales y An\'alisis Num\'erico and IMUS, Universidad de Sevilla, 
		Sevilla. Spain}
	\email{guillen@us.es}
	
	\author[C. Perusato]{Cilon F. Perusato}
	\address[C. F. Perusato]{Departamento de Matem\'atica. Universidade Federal de Pernambuco, 
		CEP 50740-560, Recife - PE. Brazil}
	\email{cilon.perusato@ufpe.br}
	
	\author[M.A. Rodr\'{\i}guez-Bellido ]{M.A. Rodr\'{\i}guez-Bellido}
	\address[M.A. Rodr\'{\i}guez-Bellido]{Dpto. de Ecuaciones Diferenciales y An\'alisis Num\'erico and IMUS, Universidad de Sevilla, 
		Sevilla. Spain}
	\email{angeles@us.es}
	
	\thanks{P. Braz e Silva was partially supported by by CAPES--PRINT - 88881.311964/2018--01, 
		CAPES-MATHAMSUD $\#$88881.520205/2020-0, and CNPq, Brazil, $\#$308758/2018-8 and $\#$432387/2018-8}
	
	\thanks{F. Guill\'en-Gonz\'alez and M.A. Rodr\'{\i}guez-Bellido acknowledges funding from 
		Grant PGC2018-098308-B-I00 (MCI/AEI/FEDER, UE), 
		Grant US-1381261 (US/JUNTA/FEDER, UE) and Grant P20$_-$01120  (PAIDI/JUNTA/FEDER, UE)} 
	
	\thanks{C. Perusato was partially supported by CAPES--PRINT - 88881.311964/2018--01 and Propesq-UFPE - 08-2019 (Qualis A). He is also grateful for the warm hospitality and the wonderful academic atmosphere during his visit at the Universidad de Sevilla, where this work was started.}
	
	
	
	\date{\today}
	
	%
	
	%
	
	%
	%
	%
	%
	\parindent=0pt
	\def\b{{\boldsymbol b}}
	\def\n{{\boldsymbol n}}

	\begin{abstract}
		An optimal control problem associated to the Keller-Segel with logistic reaction system  will be studied in $2D$ domains.  The control acts in a bilinear form only in the chemical equation.  The existence of optimal control and a necessary optimality system are deduced.
		The main novelty is that control can be rather singular and the state (cell density $u$ and the chemical concentration  $v$) remains only in a weak setting, which is not usual in the literature to solve optimal control problems subject to chemotaxis models (see e.g. \cite{guillen-mallea-rodriguez}).
	\end{abstract}
	
	\maketitle

	\section{Introduction}
	
	\subsection{The controlled model}
	In this work we study an optimal control problem for the (attractive or repulsive) 
	Keller-Segel model  in a $2D$ domain $\Omega\subset\mathbb{R}^2$ with logistic source term and  bilinear control acting on the chemical equation:
	\begin{equation}\label{KS}
		\left\{\begin{array}{rcll}
			\partial_t u -\Delta u + \kappa \nabla \cdot (u \, \nabla v) &=&
			r \, u- \mu \, u^2 & \mbox{in $\Omega \times (0,T)$,}\\
			\partial_t v -\Delta v + v &=& u + f \, v\, 1_{\Omega_c}& \mbox{in $\Omega \times (0,T)$,}\\
			\partial_\n u= \partial_\n v &=&0  & \mbox{on $\partial\Omega \times (0,T)$,}\\
			u(0,\cdot)=u_0\ge 0, \quad v(0,\cdot)&=&v_0\ge 0 & \mbox{in $\Omega$.}
		\end{array}\right.
	\end{equation}
	Here, $f:Q_c:=(0,T)\times \Omega_c\to \mathbb{R}$ is the control with $\Omega_c\subset \Omega\subset \mathbb{R}^2$ the control domain,  and the states $u,v:Q:=(0,T)\times \Omega\to \mathbb{R}_+^2$ are the cellular density and chemical concentration, respectively. 
	Moreover, $r\in \mathbb{R}$ and $\mu >0$ are coefficients of the logistic reaction, and $\kappa \in \mathbb{R}$ is the chemotaxis coefficient ($\kappa>0$ models  attraction and $\kappa<0$ repulsion). We are interested in the study of a control problem associated to the following weak solution concept  of  system (\ref{KS}). Hereafter, $L^{2+}$ means $L^{2+\varepsilon}$ for $\varepsilon $ small enough.
	\begin{definition}\label{strong}
		Let $f\in L^{2+}(Q_c):=L^{2+}(0,T;L^{2+}(\Omega_c))$, $u_0\in L^2(\Omega)$, $v_0\in W^{1+,2+}(\Omega)$	
		with $u_0\ge 0$ and $v_0\ge 0$ a.e.~in $\Omega$. A pair $(u,v)$
		is called a weak solution of problem (\ref{KS}) in $(0,T)$, if 
		\begin{eqnarray*}
			&& u\ge0, \quad v\ge 0\quad \hbox{a.e. in $Q=(0,T)\times \Omega$,} \\
			&&u\in W_2:=\{u\in C([0,T];L^2(\Omega))\cap L^2(0,T;H^1(\Omega)),\ \partial_tu\in L^2(0,T;H^1(\Omega)')\},\label{st-1}\\
			&&v\in X_{2+}:=\{v\in C([0,T];W^{1+,2+}(\Omega))\cap L^{2+}(0,T;W^{2,2+}(\Omega)),\ \partial_tv\in L^{2+}(Q)\},\label{st-11}
		\end{eqnarray*}
		the equation (\ref{KS})$_{1}$ jointly boundary condition for $u$ hold in a variational sense,  while equation (\ref{KS})$_{2}$ and the boundary condition for $v$ are satisfied  pointwisely, 
		and initial conditions (\ref{KS})$_3$ and (\ref{KS})$_4$ are satisfied in $L^2(\Omega)$ and $W^{1+,2+}(\Omega)$, respectively.
	\end{definition}
	
	Notice that, since we are in $2D$ bounded domains, $v\in C([0,T];W^{1+,2+})$ implies 
	$v\in L^\infty(0,T;L^\infty)$, 
	hence using that $f\in L^{2+}(Q_c)$ one has  $fv\in L^{2+}(Q)$. That means that $v\in X_{2+}$ is the maximal regularity which can be obtained.
	The previous weak  regularity for $u\in W_2$ will be enough to solve the following  optimal control problem. 
%
%
%
	\begin{equation}\label{func}
		\left\{
		\begin{array}{l}
			\mbox{Find  }(u,v,f)\in W_2\times X_{2+}\times\mathcal{F}\mbox{ minimizing the functional }\\
			\noalign{\vspace{-1ex}}\\
			J(u,v,f):=\displaystyle\frac{\gamma_u}{2}\int_0^T\|u(t)-u_d(t)\|^{2}_{L^{2}{(\Omega)}}dt \\
			+\displaystyle\frac{\gamma_v}{2}\int_0^T\|v(t)-v_d(t)\|^2_{L^2{(\Omega)}}dt
			+\displaystyle\frac{\gamma_f}{{2+}}\int_0^T\|f(t)\|^{2+}_{L^{2+}(\Omega_c)}dt\\
			\noalign{\vspace{-1ex}}\\
			\mbox{  subject to $(u,v,f)$ be a weak solution of the PDE system (\ref{KS}),}
		\end{array}
		\right.
	\end{equation}
	where $(u_d,\, v_d)\in L^2(Q)^2$ represent  the target states 
		and the nonnegative numbers $ \gamma_u$, $\gamma_v$ and $\gamma_f$ measure the cost of the states and control,
		respectively. With respect to the control constraint,  we assume 
		\begin{equation*}\label{F-def}
			\mathcal{F}\subset L^{2+}(Q_c) \quad \mbox{ be a
				nonempty, closed and convex set.} 
		\end{equation*}
		
		\
		
		The functional $J$ defined in (\ref{func}) describes the deviation of the cell density $u$ 
		and the chemical concentration $v$
		from a target  cell density $u_d$ and chemical concentration $v_d$, respectively; plus  the cost of the control $f$ measured in the $L^{2+}$-norm.
		
		\subsection{Previous results}
\definecolor{airforceblue}{rgb}{0.36, 0.54, 0.66}
Before continuing our discussion, let us give some motivation for the study of \eqref{KS}.
In the last decades, there has been a surge of activity on the study of the chemotaxis model which describes the movement of the cells directed by the concentration gradient of a
chemical substance in their environment. Moreover, it is important to consider the biological situation where the bacterial
population may proliferate according to a logistic law. An interesting feature in chemotaxis corresponds
to the movement of cells directed by the gradient of the chemical signal which is produced by cells themselves. When one considers also the interactions between cells and the chemical signal with liquid environments, one gets the chemotaxis-fluid system, which is basically the chemotaxis model coupled with the Navier-Stokes equations. For more details, see  the excellent review \cite{BBTW} and the references therein.  

\

Let us recall some issues  related to the uncontrolled equations \eqref{KS}, i.e., when $ f \equiv 0 $. Plenty of results have been obtained here. Amongst the many articles related to the uncontrolled sysstem, let us mention those on existence of weak and strong solutions in $ \mathbb{R}^2 $. In this case, without considering logistic reaction (i.e. $ r = \mu = 0 $), the existence of global weak solutions was provided by Liu and Lorz  \cite{Liu-Lorz-2011}. In two-dimensional bounded convex domains, the existence of (global) classical solutions was obtained by M. Winkler \cite{winkler2012}. 
In the presence of logistic source, the existence of global weak solutions (and the long time behavior) has been analyzed in \cite{Lankeit2016} by J. Lankeit. In this case, the existence of global mild solutions was examined in \cite{Rodriguez-Ferreira-Roa}. For 3D domains, we also refer \cite{winklerJFA2019} and the references therein.   
		
\
		
It is important to mention that remarkable progress has been made in mathematical and numerical analysis  of optimal
control problems for viscous flows described by the Navier-Stokes equations and other related models, see e.g., \cite{Abergel, casas, Zepeda-Torres-Roa}. 
However, the literature related to optimal control for chemotaxis problems  is still scarce.
In \cite{rodriguez-rueda-villamizar}, the authors study an optimal (distributed) control problem where the state problem is given by a stationary chemotaxis model coupled with the Navier-Stokes equations. We note that in \cite{chaves-guerrero-1, chaves-guerrero-2} the authors provide some results related to the controllability for the nonstationary Keller-Segel system and the nonstationary chemotaxis-fluid model with
consumption of chemoattractant substance associated to a chemotaxis system, based on Carleman-type estimates for the solutions of the adjoint system. 	
Recently, a bilinear optimal control problem associated to the chemotaxis-Navier-Stokes model (without logistic source) in bounded three-dimensional domains was examined in \cite{Lopez-Roa}. For the chemo-repulsion case, this problem was studied in \cite{guillen-mallea-rodriguez,guillen-mallea-rodriguez3} for $2D$ and $3D$ domains respectively, and in \cite{guillen-mallea-villamizar} for $2D$ domains with a potential nonlinear production term, that is changing the production term $u$ in the $v$ equation by $u^p$, with $p>1$.

		\subsection{Main contributions of the paper}

		First of all, we will prove the existence and uniqueness of weak solutions of (\ref{KS}) and the continuous dependence of the weak solution $(u,v)$ respect the control $f$.  
		
		\begin{theorem}\label{weak_solution}
			Let $u_0\in L^{2}(\Omega)$, $v_0\in  W^{1+,2+}(\Omega)$ with $u_0\ge 0$ and $v_0\ge 0$ in $\Omega$, and $f\in L^{2+}(Q_c)$. 
			There exists a unique weak solution $(u,v)$ of system (\ref{KS}) in sense of Definition \ref{strong}. Moreover, there exists a positive constant
			$$
			\mathcal{K}_1:=\mathcal{K}_1(r,\mu,\kappa,\vert \Omega\vert,T,\|u_0\|_{L^2},\|v_0\|_{W^{1+,2+}},\|f\|_{L^{2+}(Q_c)}),
			$$ 
			such that
			\begin{equation}\label{bound_solution}
				\|(u,v) \|_{W_2 \times X_{2+}}
				\le \mathcal{K}_1.
			\end{equation}
			where we denote 
			$$
			\|(u,v) \|_{W_2 \times X_{2+}}
			:=
			\|(\partial_tu,\partial_tv)\|_{L^2(H^1)'\times L^{2+}(L^{2+})}+\|(u,v)\|_{C(L^2\times W^{1+,2+})}
			$$
			$$
			+\|(u,v) \|_{L^2(H^1)\times L^{2+}(0,T;W^{2,2+})}
			$$
			Finally, for any $r,\mu,\kappa, \Omega,T,u_0,v_0$, the constant  $\mathcal{K}_1$ is bounded if $f$ is bounded in $L^{2+}(Q_c)$.
		\end{theorem}

		The second main result of this paper will be the existence of a global optimal solution of (\ref{func}):
		\begin{theorem}\label{existence_solution}
			Let 
			$(u_0,  v_0)\in L^2(\Omega)\times  W^{1+,2+}(\Omega)$ with $u_0\ge0$ and $v_0\ge 0$ in $\Omega$. Assuming that either $\gamma_f>0$ or $\mathcal{F}$ is bounded in
			$L^{2+}(Q_c)$, then the bilinear optimal control problem (\ref{func}) has at least one global optimal solution $(\tilde{u},\tilde{v},\tilde{f})$.
		\end{theorem}

		Finally, we obtain the existence and uniqueness of Lagrange multipliers associated to the optimal control (\ref{func}):
		
		\begin{theorem}\label{Th:Optimality-System}
			Let $\tilde{s}=(\tilde{u},\tilde{v},\tilde{f})\in\mathcal{S}_{ad}$ be a local optimal solution for the control problem (\ref{func}). Then, there exists a unique Lagrange multiplier
			$(\lambda,\eta)\in X_{2}\times W_{2}$ satisfying the  optimality system
			\begin{equation}\label{R5}
				\left\{
				\begin{array}{rcl}
					-\partial_t\lambda-\Delta\lambda+
					\kappa\,
					\nabla\lambda\cdot\nabla\tilde{v}-\eta&=&\gamma_u (\tilde{u}-u_d)\ \ \mbox{ in  Q,}\\
					-\partial_t\eta-\Delta\eta-
					\kappa \,
					\nabla\cdot(\tilde{u}\nabla\lambda)+\eta-\tilde{f}\,\eta\, 1_{\Omega_c}&=&\gamma_v(\tilde{v}-v_d)\ \ \mbox{ in  Q,}\\
					\lambda(T)=0,\ \eta(T)&=&0\quad  \mbox{ in }\Omega,\\
					\dfrac{\partial\lambda}{\partial\n}=0,\ \dfrac{\partial\eta}{\partial\n}&=&0\quad\mbox{ on }(0,T)\times\partial\Omega,\\
					\displaystyle\int_0^T\int_{\Omega_c}(\gamma_f {\rm sgn} \tilde{f} |\tilde{f}|^{1+}+\tilde{v}\,\eta)(f-\tilde{f})&\ge&0\ \ \ \forall f\in\mathcal{F}.
				\end{array}
				\right.
			\end{equation}
		\end{theorem}
		
		\begin{remark}
			If $\gamma_f>0$ and  there is no convexity constraint on the control, that is, $\mathcal{F}\equiv L^{2+}(Q_c)$, then optimality condition (\ref{R5})$_5$ becomes 
			$$
			\gamma_f {\rm sgn} \tilde{f} |\tilde{f}|^{1+} 1_{\Omega_c}+\tilde{v}\,\eta\, 1_{\Omega_c}=0.
			$$
		\end{remark}
		
		\
		
		The rest of the paper is organized as follows. In Section \ref{Se:Preliminaires} some Preliminairy results  which will be used later are introduced. The proofs of Theorems \ref{weak_solution}, \ref{existence_solution} and \ref{Th:Optimality-System} 
		are given in Sections \ref{Sec:weak},   \ref{Sec:existence} and   \ref{Sec:LM}  respectively.
		
		\section{Preliminary results}\label{Se:Preliminaires}
		Along this manuscript  the following result on $L^p$ regularity 
		 will be considered.
		\begin{theorem}[\cite{feireisl}, page 344] \label{feireisl}
			For $\Omega\in C^2$, let $1<p<3$, $u_0\in {W}^{2-2/p,p}(\Omega)$ and  $g\in L^p(Q)$. Then the problem
			\begin{equation*}
				\left\{
				\begin{array}{rcl}
					\partial_tu-\Delta u&=&g \quad\mbox{ in }Q,\\
					u(0,\cdot)&=&u_0\quad\mbox{ in }\Omega,\\
					\dfrac{\partial u}{\partial{\bf n}}&=&0\quad\mbox{ on }(0,T)\times\partial\Omega,
				\end{array}
				\right.
			\end{equation*}
			admits a unique solution $u$ such that
			\begin{equation*}
				u\in  C([0,T];{W}^{2-2/p,p})\cap L^p(W^{2,p}),
				\quad \partial_tu\in L^p(Q).
			\end{equation*}
			Moreover, there exists a positive constant $C:=C(p,\Omega,T)$ such that
			\begin{equation*}\label{des-regularity}
				\|u\|_{C({W}^{2-2/p,p})}+\|\partial_tu\|_{L^p(Q)}+\|u\|_{L^p(W^{2,p})}
				\le C(\|g\|_{L^p(Q)}+\|u_0\|_{{W}^{2-2/p,p}}).
			\end{equation*}
		\end{theorem}

		\
		
	The existence of solutions of some initial and boundary-value problems will be proven by means of:		
		\begin{theorem}[Leray-Schauder fixed-point Theorem] \label{LSFP}
			Let $\mathcal{X}$ a Banach space and $T: \, \mathcal{X} \rightarrow \mathcal{X}$ a continuous and compact operator. If the set 
			$$
			\{
			x \in \mathcal{X}: \, x=\alpha \, Tx \quad \mbox{for some $0\le \alpha \le 1$}
			\}
			$$
			is bounded, then $T$ has (at least) a  fixed point.
		\end{theorem}
		
		\bigskip
		
		In this paper, the following two compactness results will be applied.
		\begin{theorem}[Aubin-Lions lemma] (See \cite[Th\'eor\`eme 5.1, p. 58]{lions}.)\label{AL} 
			Let $\mathcal{X}$, $B$ and $\mathcal{Y}$ be reflexive Banach spaces such that $\mathcal{X} \subset B \subset \mathcal{Y}$, 
			with compact embedding $\mathcal{X}\mapsto B$ and continuous embedding $B \hookrightarrow \mathcal{Y}$.
			It is defined 
			$$
			W=\{w:\, w\in L^{p_0}(0,T;\mathcal{X})), \, \partial_t w \in L^{p_1}(0,T;\mathcal{Y})\}
			$$ 
			for a finite $T>0$ and $p_0$, $p_1 \in (1,+\infty)$. Then, the injection of $W$ into $L^{p_0}(0,T;B)$ is compact.
		\end{theorem}
		
		\begin{theorem}[Simon's compactness result] (See \cite[Corollary 4]{simon}.) \label{Simon-t}
			Let $\mathcal{X}$, $B$ and $\mathcal{Y}$ be Banach spaces such that $\mathcal{X} \subset B \subset \mathcal{Y}$, 
			with compact embedding $\mathcal{X}\mapsto B$ and continuous embedding $B \hookrightarrow \mathcal{Y}$. 
			Let $F$ be a bounded set in $L^{\infty}(0,T;\mathcal{X})$ such that the set  
			$\partial_t F=\{ \frac{\partial f}{\partial t}; \, f\in F \}$ 
			is bounded in $L^r(0,T;\mathcal{Y})$ for some $r>1$. Then $F$ is relatively compact in $C([0,T];B)$. 
		\end{theorem}

		\section{Proof of Theorem \ref{weak_solution}} \label{Sec:weak}

		The proof of Theorem \ref{weak_solution} will be made in the next two subsections. For the existence, we  use the Leray-Schauder fixed point theorem. The uniqueness is get by a comparison argument.

		\subsection{Existence}\label{existence}
		
		Let us  introduce the spaces
		\begin{equation}\label{pf}
			\mathcal{X}_u:= L^{4-}(Q)\quad \mbox{ and }\quad
			\mathcal{X}_v:=  L^{\infty}(Q),
		\end{equation}
		and the operator
		$R:\mathcal{X}_u\times \mathcal{X}_v\rightarrow W_2 \times X_{2+}\hookrightarrow \mathcal{X}_u\times \mathcal{X}_v$ defined by $R(\bar{u},\bar{v})=(u,v)$ is the solution of the decoupled  linear problem
		\begin{equation}\label{pf-1}
			\left\{
			\begin{array}{l}
				\displaystyle
				\int_0^T \langle \partial_tu,\varphi\rangle 
				+ \int_0^T\int_\Omega  \nabla u \cdot \nabla \varphi + \mu \, \bar{u}_+\, u\, \varphi 
				\\ \displaystyle
				\quad =
				\int_0^T\int_\Omega  r \, \bar{u}_+ \varphi 
				- \kappa\int_0^T\int_\Omega \bar {u}_+\nabla v \cdot \nabla \varphi,
				\quad \forall\, \varphi \in L^2(H^1),
				\\
				\\
				\partial_tv-\Delta v+v = \bar{u}_+ + f\, \bar{v}_+ \, 1_{\Omega_c}
				\quad \hbox{in $Q,$}
				\\
				u(0)=u_0,\ v(0)=v_0, \quad \hbox{in $\Omega,$}\\
				\dfrac{\partial v}{\partial\n}=0, \quad \hbox{on $(0,T)\times \partial\Omega,$}
			\end{array}
			\right.
		\end{equation}
		where $\bar{u}_+:=\max\{\bar{u},0\}\ge0$, $\bar{v}_+:=\max\{\bar{v},0\}\ge0$. In fact, first we compute $v$ and after $u$. 
		In the following lemmas we will prove the hypotheses  of Leray-Schauder fixed point theorem.
		\begin{lemma}\label{compact}
			The operator $R:\mathcal{X}_u\times\mathcal{X}_v\rightarrow \mathcal{X}_u\times\mathcal{X}_v$ is well defined and compact.
		\end{lemma}
		\begin{proof}
			Since  $f\in L^{2+}(Q_c)$ and $\bar{v}\in L^{\infty}(Q)$ implies  $f \bar{v}\in L^{2+}(Q)$, hence there exists a unique  $v\in X_{2+} $ solution of the $v$-problem in \eqref{pf-1}.
			By considering the 	linear parabolic $u$-problem in \eqref{pf-1}, 
			one has $u\in W_2$ owing to $v\in X_{2+}$, hence $\nabla v\in L^{4+}(Q)$ and then 
			$\bar {u}_+\nabla v \in L^2(Q)$.
			Finally, since $R$ maps bounded sets of $\mathcal{X}_u\times \mathcal{X}_v$ into bounded sets of  $W_2 \times X_{2+}$, then $R$ is compact from $\mathcal{X}_u\times \mathcal{X}_v$ to itself. 
		\end{proof}
		\begin{lemma}\label{cot-1}
			The set
			\begin{equation}\label{cot-2}
				T_\alpha=\{(u,v)\in W_2\times X_{2+}\,:\, (u,v)=\alpha R(u,v)\mbox{ for some }\alpha\in[0,1]\}
			\end{equation}
			is bounded in $\mathcal{X}_u\times\mathcal{X}_v$
			(independently of $\alpha \in [0,1]$).  In fact, $T_\alpha$ is also bounded in $ W_2\times X_{2+}$, because there exists 
			\begin{equation}\label{cot-3}
				M=M(r,\mu,\kappa ,\vert \Omega\vert,T,\|u_0\|_{L^2},\|v_0\|_{W^{1^+,2^+}},\|f\|_{L^{2+}(Q_c)})>0,
			\end{equation}
			with $M$ independent of $\alpha$, such that 
			all pairs of functions $(u,v)\in T_\alpha$ for $\alpha\in[0,1]$ satisfy
			\begin{equation}\label{cot-4}
				\|(u,v) \|_{W_2 \times X_{2+}} \le M.
			\end{equation}
		\end{lemma}
		
		\begin{proof}
			Let $(u,v)\in T_\alpha$ for $\alpha\in (0,1]$ (the case  $\alpha=0$  is trivial). Then, 
			owing to Lemma \ref{compact}, $(u,v) \in W_2 \times X_{2+}$ and satisfies the following problem:	
			\begin{equation}\label{KSa}
				\left\{
				\begin{array}{l}
					\displaystyle
					\int_0^T \langle \partial_tu,\varphi\rangle 
					+ \int_0^T\int_\Omega  \nabla u \cdot \nabla \varphi + \mu \,{u}_+\, u\, \varphi 
					\\ \displaystyle
					\quad =
					\alpha \int_0^T\int_\Omega  r \,{u}_+ \varphi 
					- \kappa\int_0^T\int_\Omega {u}_+\nabla v \cdot \nabla \varphi,
					\quad \forall\, \varphi \in L^2(H^1),
					\\
					\\
					\partial_tv-\Delta v+v = \alpha\,{u}_+ + \alpha \, f\,{v}_+ \, 1_{\Omega_c}
					\quad \hbox{a.e. in $Q,$}
				\end{array}
				\right.
			\end{equation}
			endowed with the corresponding initial and boundary conditions. Therefore, it suffices to look for a bound of $(u,v)$ in $W_2 \times X_{2+}$ independent of $\alpha$. 
			This bound is carried out into six steps:
			
			\
			
			{{\bf Step 1:}} \ Non-negativity: $u,v\ge 0$.

			By taking in (\ref{KSa})$_1$  $\varphi=u_-:=\min\{u,0\}\le0$ (that is possible because $u\in L^2(H^1)$),
			and considering that $u_-=0$ if $u\ge 0$, $\nabla u_-=\nabla u$ if $u\le 0$, and $\nabla u_-=0$ if $u>0$, we have
			$$
			\frac12\frac{d}{dt}\|u_-\|^2+\|\nabla u_-\|^2= \kappa(u_+\nabla v,\nabla u_-)
			+\alpha \, r \, (u_+,u_-) -  \mu \, ((u_+)^2,u_-)
			=0,
			$$
			thus $u_-\equiv 0$ and, consequently, $u\ge0$.  Similarly, testing  (\ref{KSa})$_2$ by $v_-:=\min\{v,0\}\le0$ 
			we obtain
			$$
			\frac12\frac{d}{dt}\|v_-\|^2+\|\nabla v_-\|^2+\|v_-\|^2=\alpha(u_+,v_-)+\alpha(fv_+,v_-)_{\Omega_c}\le 0,
			$$
			which implies $v_-\equiv0$, then $v\ge0$. Therefore $(u_+,v_+)=(u,v)$. In particular, $(u,v,f)$ is also the solution of problem \eqref{KSa} changing $u_+$ by $u$ and $v_+$ by $v$. Therefore, fixed-point of $R$ are in particular weak solutions of problem (\ref{KS}).
				
				\
				
				{\bf Step 2:} Boundedness of  $\displaystyle\int_\Omega u(x,t) \, dx $ . 
				
				By taking $\varphi=1$ in  (\ref{KSa})$_1$, we obtain:
				\begin{equation}\label{eq1}
					\displaystyle\frac{d}{dt} 
					\displaystyle\int_{\Omega} u(x,t) \, dx
					+  \mu \, \displaystyle\int_{\Omega} u^2(x,t) \, dx=  \alpha \, r \, \displaystyle\int_{\Omega} u(x,t) \, dx
				\end{equation}
				
				Using Cauchy-Schwartz inequality, we have:
				$$
				\displaystyle\int_{\Omega} u(x,t) \, dx
				\le
				\vert \Omega \vert^{1/2} \, 
				\left(\displaystyle\int_{\Omega} u^2(x,t) \, dx
				\right)^{1/2},
				$$
				which from (\ref{eq1}) let us deduce that:
				\begin{equation}\label{eq2}
					\displaystyle\frac{d}{dt} 
					\displaystyle\int_{\Omega} u(x,t) \, dx
					+ \displaystyle\frac{\mu}{\vert \Omega \vert} \, 
					\left(\displaystyle\int_{\Omega} u(x,t) \, dx\right)^2
					\le
					\alpha \, r \, \displaystyle\int_{\Omega} u(x,t) \, dx.
				\end{equation}
				Using the change of variable $y(t)=\displaystyle\int_{\Omega} u(x,t) \, dx$, (\ref{eq2}) becomes:
				\begin{equation}\label{eq3}
					y'(t) + \displaystyle\frac{ \mu}{\vert \Omega \vert} \,  y(t)^2 \le \alpha \, r \, y(t)\le  r \, y(t) 
				\end{equation}
				which is related to a Bernouilli ODE. 
				Making use of $z(t)=y(t)^{-1}$, we can deduce that:
				$$
				-z'(t) + \displaystyle\frac{\mu}{\vert \Omega \vert}  \le  r \, z(t),
				$$ 
				and thus
				\begin{equation}\label{eq4}
					z'(t)+  r \, z(t) \ge  \displaystyle\frac{\mu}{\vert \Omega \vert} .
				\end{equation}
				This inequality is equivalent to:
				$$
				\displaystyle\frac{d}{dt} \left( e^{ rt} \, z(t) \right) 
				\ge 
				\displaystyle\frac{\mu}{\vert \Omega \vert} \, e^{  rt}
				= 
				\displaystyle\frac{\mu}{\vert \Omega \vert} \, \displaystyle\frac{d}{dt} \left(e^{  rt} \right)
				$$
				and therefore
				\begin{equation}\label{eq5}
					z(t) \ge z(0) \, e^{- rt} + \displaystyle\frac{\mu}{r \vert \Omega \vert} \, \left(
					1-e^{- rt}
					\right)
					= \displaystyle\frac{\mu}{r\vert \Omega \vert} 
					+ 
					\left(z(0) -\displaystyle\frac{\mu}{r\vert \Omega \vert}\right) 
					e^{- rt}.
				\end{equation}
				Now, we consider two cases:
				\begin{itemize}
					
					\item if $z(0) \ge \displaystyle\frac{\mu}{r\vert \Omega \vert}$ (i.e., $y(0) \le \displaystyle\frac{r\vert \Omega \vert}{\mu}$\Big), then:
					$$
					z(t)
					\ge \displaystyle\frac{\mu}{r\vert \Omega \vert},
					$$
					which implies that:
					\begin{equation}\label{eq61}
						y(t) \le \displaystyle\frac{r\vert \Omega \vert}{\mu}, \qquad \forall t\ge 0 \qquad
						\mbox{(independently of $\alpha$)}.
					\end{equation}
					
					\item if $z(0) \le \displaystyle\frac{\mu}{r\vert \Omega \vert}$ (i.e., $y(0) \ge \displaystyle\frac{r\vert \Omega \vert}{\mu}$\Big), then from (\ref{eq5}) we can deduce that:
					$$
					z(t)
					\ge z(0) \, e^{- rt} + \displaystyle\frac{\mu}{r \vert \Omega \vert} \, \left(
					1-e^{- rt}
					\right) \ge z(0) \, e^{- rt} + z(0) \, \left(
					1-e^{- rt}
					\right) = z(0)
					$$
					and therefore
					\begin{equation}\label{eq62}
						y(t) \le y(0)=\displaystyle\int_{\Omega} u_0(x) \, dx
						=m_0, \qquad \forall t\ge 0
						\qquad
						\mbox{(independently of $\alpha$)}.
					\end{equation}
					As a conclusion, from (\ref{eq61}) and (\ref{eq62}), we arrive at the bound
					\begin{equation}\label{eq-conservation}
						y(t)=\displaystyle\int_{\Omega} u(x,t) \, dx
						\le \max 
						\left\{
						m_0 
						,\displaystyle\frac{r\vert \Omega \vert}{\mu}
						\right\}:=K_1, \quad \forall t\ge 0.
					\end{equation}
					
				\end{itemize}
				
				\
				
				{\bf Step 3:} Bound of $ u$ in $L^2(0,T;L^2(\Omega))$
				
				Integrating directly in $(0,T)$ for a fixed $T$ in (\ref{eq1}), and using \eqref{eq-conservation}, we obtain that:
				\begin{equation}\label{u2}
					\displaystyle\int_0^T \int_{\Omega} u^2(x,t) \, dx \, dt 
					\le 
					\displaystyle\frac{m_0+\alpha \, r \, K_1 \, T}{\mu} 
					\le
					\displaystyle\frac{m_0+r \, K_1 \, T}{\mu}:=K_2(T),
				\end{equation}
				which implies that
				\begin{equation}\label{u2-bis}
					\Vert  u \Vert_{L^2(Q)}^2 
					\le 
					\displaystyle
					K_2(T).
				\end{equation}
				
				\
				
				{\bf Step 4:} Bound of $v$ in $L^\infty(0,T;H^1(\Omega))\cap L^2(0,T;H^2(\Omega))$
				
				Taking $v$ as test function in (\ref{KSa})$_2$, we obtain:
				\begin{equation}\label{v1}
					\begin{array}{rcl}
						\displaystyle\frac{1}{2} \frac{d}{dt} \Vert v \Vert_{L^2}^2 +
						\Vert v \Vert_{H^1}^2 
						&= & \alpha \, \displaystyle\int_{\Omega} u \,  v \, dx +
						\alpha \, 
						\displaystyle\int_{\Omega_c} f \, v^2 \, dx 
						\\
						\noalign{\vspace{-1ex}}\\
						\mbox{because $(\alpha \in (0,1])$} & \le & 
						\Vert   u \Vert_{L^2} \Vert  v \Vert_{L^2}
						+
						\Vert f \Vert_{L^2} \Vert v \Vert_{L^4}^2 \\
						\noalign{\vspace{-1ex}}\\
						& \le &  \Vert  u \Vert_{L^2} \Vert v \Vert_{L^2}
						+
						\Vert f \Vert_{L^2} \Vert v \Vert_{L^2} \Vert v \Vert_{H^1} \\
						\noalign{\vspace{-1ex}}\\
						& \le &  
						\delta \, 
						\Vert v \Vert_{H^1}^2
						+ 
						C_{\delta} \, \left(
						\Vert u \Vert_{L^2}^2
						+
						\Vert f \Vert_{L^2}^2 \Vert v \Vert_{L^2}^2
						\right)
					\end{array}
				\end{equation}
				where we used the following standard inequality in 2D domains 
				$$\|u\|_{L^{4}} \leq C\|u\|_{L^2}^{1 / 2}\|u\|_{H^{1}}^{1 / 2}, \quad \forall u \in H^{1}(\Omega).$$
				
				Therefore, by taking $\delta$ small enough,
				\begin{equation}\label{v1-1}
					\displaystyle \frac{d}{dt} \Vert v \Vert_{L^2}^2 +
					\Vert v \Vert_{H^1}^2 
					\le   
					C \, \left(
					\Vert u \Vert_{L^2}^2
					+
					\Vert f \Vert_{L^2}^2 \Vert v \Vert_{L^2}^2
					\right)	
				\end{equation}
				From Gronwall's lemma, 
				and thanks to the boundedness of $u$ and $f$ in $L^2(Q)$,
				one has $v$ bounded in 
				$L^\infty(0,T;L^2(\Omega))\cap L^2(0,T;H^1(\Omega))$.

				\
				
				By taking $-\Delta v$ as test function in (\ref{KSa})$_2$, we obtain:
				\begin{equation}\label{v2}
					\begin{array}{l}
						\displaystyle\frac{1}{2} \frac{d}{dt} \Vert \nabla v \Vert_{L^2}^2 +
						\Vert \Delta v \Vert_{L^2}^2 + \Vert \nabla v  \Vert_{L^2}^2 
						= -\alpha \, \displaystyle\int_{\Omega}
						u \,  \Delta v
						\, dx -
						\alpha \, 
						\displaystyle\int_{\Omega}
						f \, v \, \Delta v 
						\, dx\\
						\noalign{\vspace{-1ex}}\\
						\le  \Vert   u \Vert_{L^2} \Vert  \Delta v \Vert_{L^2}
						+
						\Vert f \Vert_{L^{2+}} \Vert v \Vert_{H^1} \Vert \Delta v \Vert_{L^2} \\
						\noalign{\vspace{-1ex}}\\
						\le   
						\delta \, \left(
						\Vert \Delta v \Vert_{L^2}^2
						+
						\Vert v \Vert_{H^1}^2
						\right)
						+
						C_{\delta} \, \left(
						\Vert   u \Vert_{L^2}^2
						+
						\Vert f \Vert_{L^{2+}}^2 \Vert v \Vert_{H^1}^2
						\right)
					\end{array}
				\end{equation}
				
				Adding (\ref{v1}) to (\ref{v2}), we obtain:
				\begin{equation}\label{v3}
					\begin{array}{rcl}
						\displaystyle \frac{d}{dt} \Vert v \Vert_{H^1}^2 +
						\Vert v \Vert_{H^2}^2 
						& \le  & 
						C \, \Big(
						\Vert  u \Vert_{L^2}^2 
						+
						\left(\Vert f \Vert_{L^2}^2 +
						\Vert f \Vert_{L^{2+}}^2
						\right) \,  \Vert v \Vert_{H^1}^2
						\Big)
					\end{array}
				\end{equation}
				
				From Gronwall's lemma, one has $v$ bounded in 
				$L^\infty(0,T;H^1(\Omega))\cap L^2(0,T;H^2(\Omega))$.
				
				\
				
				{\bf Step 5:} $u$ is bounded in $L^\infty(0,T;L^2(\Omega))\cap L^2(0,T;H^1(\Omega))$.
				
				By testing (\ref{KSa})$_1$ by $ u $, after a few computations, we get,      
				\begin{equation*}
					\begin{split}
						\frac{1}{2} \frac{d}{dt} \| u \|_{L^2}^2 + \|\nabla u \|_{L^2}^2 
						+ \mu \, \Vert u \Vert_{L^3}^3
						+ \| u \|_{L^2}^2 
						\\ 
						\leq \kappa\, \|u\|_{L^{4}}\|\nabla v\|_{L^{4}}\|\nabla u\|_{L^2} + (r_+\alpha +1) \|u \|_{L^2}^2 \\
						\leq  C\|u\|_{L^2}^{2}\|\nabla v\|_{L^{4}}^{4}+\frac{1}{2}\|u\|_{H^{1}}^{2} +(r_+ +1) \|u \|_{L^2}^2, \\
					\end{split}
				\end{equation*} 
				Then, we arrive at
				\begin{equation}\label{st5_1}
					\begin{split}
						\frac{\mathrm{d}}{\mathrm{d} t}\|u\|_{L^2}^{2}+\|u\|_{H^{1}}^{2} 
						\leq 
						C\|\nabla v\|_{L^{4}}^{4}\|u\|_{L^2}^{2}+	
						2 \, (r_+ +1) \|u \|_{L^2}^2 
					\end{split}
				\end{equation}
				Therefore, applying the Gronwall lemma and using Step 4, we obtain that $ u $ is  bounded  in 
				$ L^{\infty}(0, T ; L^{2}(\Omega)) \cap L^{2}(0, T ; H^{1}(\Omega)) $.
				
				\
				
				{\bf Step 6:} $v$ is bounded in $L^\infty(0,T;W^{1+,2+}(\Omega))\cap L^{2+}(0,T;W^{2,2+}(\Omega))$.
				From {\bf Step 4} we can deduce that $v \in L^{\infty-}(Q)$, and from {\bf Step 5}, we can deduce that $u \in L^4(Q)$. Therefore, $u+f \, v \in L^{2+}(Q)$. Then, heat regularity result in Theorem \ref{feireisl}, allows us to deduce that $v \in X_{2^+}$ 
				and the corresponding bound on $X_{2^+}$ depending on $\Vert v_0 \Vert_{W^{1^+,2^+}(\Omega)}$ and the bound of $u+f \, v $ in $L^{2+}(Q)$.

				\
				
				This finishes Lemma  \ref{cot-1}.	
			\end{proof}
			
			
			%
			%
			\begin{lemma}\label{conti-operator}
				The operator 
				$R:\mathcal{X}_u\times \mathcal{X}_v\rightarrow \mathcal{X}_u\times \mathcal{X}_v$, defined in
				(\ref{pf-1}), is continuous.
			\end{lemma}
			The proof is similar to Lemma 3.4 in \cite{guillen-mallea-rodriguez}.
			
			\
			
			Consequently, from Lemmas  {\ref{compact}}, \ref{cot-1} and \ref{conti-operator}, 
			it follows that the operator $R$  satisfy
			the conditions of the Leray-Schauder fixed-point theorem. Thus, we conclude that the map $R(\bar{u},\bar{v})$ has at least a  fixed point, 
			$R(u,v)=(u,v)$, which is a weak solution to system (\ref{KS}) in $(0,T)$.
			
			Finally, we observe that estimate (\ref{bound_solution})  follows the same steps giving in the proof of Lemma  \ref{cot-1} (now for the case $\alpha=1$).
			

			\subsection{Uniqueness of solution}
			This proof follows the same argument than in \cite{guillen-mallea-rodriguez}, but  it is included here for reader convenience.  
			Let $(u_1,v_1),\, (u_2,v_2)\in W_2\times X_{2}$ two weak solutions of system (\ref{KS}). 
			Substracting  equations (\ref{KS}) for
			$(u_1,v_1)$ and $(u_2,v_2)$, and  denoting
			$u:=u_1-u_2$ and $v:=v_1-v_2$, we obtain the following system
			\begin{equation}\label{uni-1}
				\left\{
				\begin{array}{rcl}
					\partial_tu-\Delta u
					+ \kappa\, \nabla\cdot(u_1\nabla v+u\nabla v_2)&=&
					r \, u -\mu \, u \, (u_1+u_2)
					\ \mbox{ in }Q,\\
					\partial_tv-\Delta v+v&=&u+f\,v\, 1_{\Omega_c}\ \mbox{ in }Q,\\
					u(0,\cdot)&=&0,\ v(0,\cdot)=0\ \mbox{ in }\Omega,\\
					\dfrac{\partial u}{\partial {\bf n}}&=&0,\ \dfrac{\partial v}{\partial{\bf n}}=0\ \mbox{ on }(0,T)\times\partial\Omega.
				\end{array}
				\right.
			\end{equation}
			Testing (\ref{uni-1})$_1$ by $u\in L^2(H^1)$ and (\ref{uni-1})$_2$ by $-\Delta v\in L^2(Q)$  we have
			\begin{equation}\label{uni-2}
				\begin{array}{rcl}
					\displaystyle \frac12 \frac{d}{dt}\left(\|u\|^2+\|\nabla v\|^2\right)
					& + & \|\nabla u\|^2+\|\Delta v\|^2+\|\nabla v\|^2
					+ \mu \, \displaystyle\int_{\Omega} u^2 \, (u_1+u_2) \, dx
					\\
					\noalign{\vspace{-1ex}}\\
					& = & r \, \Vert u \Vert^2 
					+ \kappa\, (u_1\nabla v+ u\nabla v_2,\nabla  u) +(u+fv,-\Delta v).
				\end{array}
			\end{equation}
			The term $ \mu \, \displaystyle\int_{\Omega} u^2 \, (u_1+u_2) \, dx$ has the good sign.
			
			\
			
			Applying the H\"{o}lder and Young inequalities, we obtain 
			\begin{eqnarray}
				(u_1\nabla v,\nabla u)&\le&\|u_1\|_{L^4}\|\nabla v\|_{L^4}\|\nabla u\|\le C\|u_1\|_{L^4}\|\nabla v\|^{1/2}\|\nabla v\|^{1/2}_{H^1}\|\nabla u\|\nonumber\\
				&\le&\delta(\|\nabla v\|^2_{H^1}+\|\nabla u\|^2)+C_\delta
				\|u_1\|^4_{L^4}\|\nabla v\|^2,
				\label{uni-3}\\
				(u\nabla v_2,\nabla u)&\le&\|u\|_{L^4}\|\nabla v_2\|_{L^4}\|\nabla u\|\le C\|u\|^{1/2}\|u\|_{H^1}^{1/2}\|\nabla v_2\|_{L^4}\|\nabla u\|\nonumber\\
				&\le&\delta\|u\|^2_{H^1}+C_\delta\|\nabla v_2\|^4_{L^4}\|u\|^2,\label{uni-4}\\
				(u,-\Delta v)&\le&\delta\|\Delta v\|^2+C_\delta\|u\|^2,\label{uni-5}\\
				(fv,-\Delta v)&\le& \Vert f \Vert_{L^{2+}}
				\Vert v \Vert_{H^{1}}
				\Vert \Delta v \Vert_{L^2} \le
				\delta\|v\|^2_{H^2}+C_\delta\|f\|^{2}_{L^{2+}}\|v\|^2_{H^1}
				.\label{uni-6}
			\end{eqnarray}

			\bigskip

			Adding (\ref{v1}) to (\ref{uni-2}), and using (\ref{uni-3})-(\ref{uni-6}), we obtain:
			\begin{equation}\label{v7}
				\begin{array}{l}
					\displaystyle  \frac{d}{dt}\left(\|u\|^2+\|\nabla v\|^2\right)
					+ \Vert \nabla u \Vert^2 +
					\| \nabla v\|_{H^1}^2
					\\
					\quad \le 
					C \, \Big(
					\Vert u \Vert^2
					+
					\|u_1\|^4_{L^4}\|\nabla v\|^2 
					+ 
					\|\nabla v_2\|^4_{L^4}\|u\|^2
					+
					\|f\|^{2}_{L^{2+}}\|v\|^2_{H^1}
					\Big)
				\end{array}
			\end{equation}
			In order to consider the completed norm for $v$, we take $v$ as test function in (\ref{uni-1})$_2$, we obtain:
			$$
			\displaystyle\frac{1}{2} \, \frac{d}{dt}
			\left(
			\|v\|^2
			\right) 
			+  \| v \|_{H^1}^2
			= \displaystyle\int_{\Omega} u \, v \, dx
			+
			\int_{\Omega_c} f \, v^2 \, dx
			$$
			which implies:
			\begin{equation}\label{v10}
				\displaystyle \, \frac{d}{dt}
				\|v\|_{L^2}^2 + \| v \|_{H^1}^2
				\le
				C 
				\left(
				\Vert u \Vert_{L^2}^2 + 
				\|f\|^{2}_{L^{2}}\|v\|^2_{L^2}
				\right)
			\end{equation}

			Adding (\ref{v7}) to (\ref{v10}), we obtain:
			\begin{equation}\label{v11}
				\begin{array}{l}
					\displaystyle\frac{d}{dt}\left(
					\|u\|^2+\|v\|_{H^1}^2
					\right)
					+ \Vert u \Vert_{H^1}^2 +
					\| v\|_{H^2}^2
					\\
					\quad \le 
					C \, \Big(\Vert u \Vert^2
					+\|u_1\|^4_{L^4}\|\nabla v\|^2 
					+ \|\nabla v_2\|^4_{L^4}\|u\|^2
					+\|f\|^{2}_{L^{2+}}\|v\|^2_{H^1}
					\Big)
				\end{array}
			\end{equation}
			Since $\|u_1\|^4_{L^4}+ \|\nabla v_2\|^4_{L^4}
			+  \|f\|^{2}_{L^{2+}} \in L^1(0,T)$ and $u_0=v_0=0$, then Gronwall lemma implies uniqueness.

			%
			%
			%
			%
			%
			
			
			\bigskip
			
			Thus, the proof of Theorem \ref{weak_solution} is finished.

				\section{Proof of Theorem \ref{existence_solution}}  \label{Sec:existence}

				The  admissible set for the optimal control problem (\ref{func}) is defined by
				\begin{equation*}\label{adm}
					\mathcal{S}_{ad}=\{s=(u,v,f)\in W_2\times X_{2+}\times\mathcal{F}\,:\, s \mbox{ is a weak solution of  (\ref{KS}) in }(0,T)\}.
				\end{equation*}
				%
				%
				From  Theorem \ref{weak_solution} one has  $\mathcal{S}_{ad}\neq\emptyset$. Let  $\{s_m\}_{m\in\mathbb{N}}:=\{(u_m,v_m,f_m)\}_{m\in\mathbb{N}}\subset\mathcal{S}_{ad}$ be a minimizing sequence
				of $J$, that is, $\displaystyle\lim_{m\rightarrow+\infty}J(s_m)=\inf_{s\in\mathcal{S}_{ad}}J(s)$. Then, by definition of $\mathcal{S}_{ad}$, for each $m\in\mathbb{N}$, $s_m$ satisfies  system
				$(\ref{eq1})_1$ variationally in $L^2((H^1)')$ and $(\ref{eq1})_2$ a.e. $(t,x)\in Q$.
				
				From the definition of $J$ and the assumption $\gamma_f>0$  or  $\mathcal{F}$ is bounded in $L^{2+}(Q_c)$, it follows that 
				\begin{equation}\label{bound_F}
					\{f_m\}_{m\in\mathbb{N}}\mbox{ is bounded in }L^{{2+}}(Q_c)
				\end{equation}
				
				From (\ref{cot-3})-(\ref{cot-4}) there exists a positive  constant $C$, independent of $m$, such that
				\begin{equation}\label{bound_u_v}
					\|(u_m,v_m)\|_{W_2\times X_{2+}}\le C.
				\end{equation}
				Therefore, from (\ref{bound_F}), (\ref{bound_u_v}), and taking into account that $\mathcal{F}$ is a closed convex subset of $L^{2+}(Q_c)$
				(hence is weakly closed in $L^{2+}(Q_c)$),  it is deduced that there exists 
				$\tilde{s}=(\tilde{u},\tilde{v},\tilde{f})\in W_2\times X_{2+}\times\mathcal{F}$
				such that, for some subsequence of $\{s_m\}_{m\in\mathbb{N}}$, still denoted by  $\{s_m\}_{m\in\mathbb{N}}$, 
				the following convergences hold, as $m\rightarrow+\infty$:
				\begin{eqnarray}
					u_m&\rightarrow&\tilde{u}\quad \mbox{weakly in }L^{2}(H^{1})\mbox{ and weakly* in }L^\infty(L^{2}),\label{c2}\\
					v_m&\rightarrow&\tilde{v} \quad\mbox{weakly in }L^{2+}(W^{2,2+})\mbox{ and weakly* in }L^\infty(W^{1+,2+}),\label{c3}\\
					\partial_tu_m&\rightarrow&\partial_t\tilde{u} \quad\mbox{weakly in } L^2((H^1)'),\label{c4}\\
					\partial_tv_m&\rightarrow&\partial_t\tilde{v} \quad\mbox{weakly in } L^{2+}(Q),\label{c5}\\
					f_m&\rightarrow&\tilde{f} \quad \mbox{weakly in } L^{2+}(Q_c),\mbox{ and }\tilde{f}\in \mathcal{F}.\label{c6}
				\end{eqnarray}
				From (\ref{c2})-(\ref{c5}), 
				and using Sobolev embedding and Aubin-Lions compactness results,  one has
				\begin{eqnarray}
					(u_m,v_m) &\rightarrow& (\tilde{u},\tilde{v}) \quad\mbox{strongly in } C^0([0,T];((H^1(\Omega))'\times L^2(\Omega)) \label{c8-a}\\
					v_m &\rightarrow&\tilde{v} \quad\mbox{strongly in } L^{\infty}(Q)). \label{c8-b}\\
					(u_m, \nabla v_m)&\rightarrow&(\tilde{u},\nabla\tilde{v})\quad\mbox{strongly in } L^{4-}(Q)) \times L^{4+}(Q).\label{c8-c}
				\end{eqnarray}
				In particular, using (\ref{c6}), (\ref{c8-b}) and (\ref{c8-c}) the limit of the nonlinear terms of (\ref{eq1}) can be controlled as follows 
				\begin{eqnarray}
					u_m\cdot\nabla v_m &\rightarrow& \tilde{u}\cdot\nabla\tilde{v} \quad \hbox{ strongly in $L^{2}(Q)$},\label{c8-0}\\
					f_m v_m 1_{\Omega_c}&\rightarrow& \tilde{f}\,\tilde{v}\, 1_{\Omega_c} \quad \hbox{ weakly in $ L^{2+}(Q)$}.\label{c8-1}
				\end{eqnarray}
				Moreover, from (\ref{c8-a}), $(u_m(0),v_m(0))$ converges to 
				$(\tilde{u}(0),\tilde{v}(0))$ in $H^1(\Omega)'\times L^2(\Omega)$, 
				and since $u_m(0)=u_0$, $v_m(0)=v_0$,  it is deduced that $\tilde{u}(0)=u_0$ 
				and $\tilde{v}(0)=v_0$.
				Thus, $\tilde{s}$ satisfies the initial conditions given in (\ref{KS}). 
				Therefore, considering the convergences (\ref{c2})-(\ref{c8-1}), 
				and taking the limit in equation (\ref{KSa}) replacing $(u,v,f)$ by $(u_m,v_m,f_m)$,
				as $m$ goes to $+\infty$,
				it is possible to conclude that $\tilde{s}=(\tilde{u},\tilde{v},\tilde{f})$ is  a weak solution of the system (\ref{KS}), that is, $\tilde{s}\in\mathcal{S}_{ad}$. Therefore,
				\begin{equation}\label{op20}
					\lim_{m\rightarrow+\infty}J(s_m)=\inf_{s\in\mathcal{S}_{ad}}J(s)\le J(\tilde{s}).
				\end{equation}
				Additionally, since $J$ is lower semicontinuous on $\mathcal{S}_{ad}$,  one has $J(\tilde{s})\le \displaystyle\liminf_{m\rightarrow+\infty} J(s_m)$, which jointly to  
				(\ref{op20}), implies that $\tilde{s}$ is a global optimal control.
				
				\section{Proof of Theorem \ref{Th:Optimality-System}}   \label{Sec:LM}
				
				\subsection{A generic Lagrange multipliers theorem}
				
				\
				
				We introduce a Lagrange multipliers theorem given by J.Zowe and S.Kurcyusz \cite{zowe}  (see also \cite[Chapter 6]{troltz}, for more details) that we will apply to get first-order necessary optimality conditions for
				a local optimal solution $(\tilde{u},\tilde{v},\tilde{f})$ of problem (\ref{func}).  
				First, we consider the following (generic) optimization problem:
				\begin{equation}\label{abs1}
					\min_{s\in \mathbb{M}} J(s)\ \mbox{ subject to }G(s)=0,
				\end{equation}
				where $J:\mathbb{X}\rightarrow\mathbb{R}$ is a functional, $G:\mathbb{X}\rightarrow \mathbb{Y}$ is an operator,
				$\mathbb{X}$ and $\mathbb{Y}$ are Banach spaces, and  $\mathbb{M}$ is a nonempty closed and convex subset of $\mathbb{X}$. 
				The corresponding admissible set  for problem (\ref{abs1}) is 
				$$
				\mathcal{S}=\{s\in \mathbb{M}\,:\, G(s)=0\}.
				$$
				
				\begin{definition}(Lagrangian)\label{lagrangian}
					The functional $\mathcal{L}:\mathbb{X}\times{\mathbb{Y}}'\rightarrow\mathbb{R}$, given by
					\begin{equation}\label{lagr-1}
						\mathcal{L}(s,\xi)=J(s)-\langle\xi,G(s)\rangle_{\mathbb{Y}'}
					\end{equation}
					is called the Lagrangian functional related to problem (\ref{abs1}).
				\end{definition}
				
				\begin{definition}(Lagrange multiplier)\label{abs2}
					Let $\tilde{s}\in\mathcal{S}$ be a local  optimal solution for problem (\ref{abs1}). Suppose that $J$ and $G$ are Fr\'echet differentiable in $\tilde{s}$. 
					Then, any $\xi\in \mathbb{Y}'$ is called a Lagrange multiplier for (\ref{abs1}) at the point $\tilde{s}$ if
					\begin{equation}\label{abs3}
							\mathcal{L}_s'(\tilde{s},\xi)[r]=J'(\tilde{s})[r]-\langle\xi, G'(\tilde{s})[r]\rangle_{\mathbb{Y}'}\ge0\quad \forall r\in \mathcal{C}(\tilde{s}),
					\end{equation}
					
					
					where $\mathcal{C}(\tilde{s})=\{\theta(s-\tilde{s})\,:\, s\in \mathbb{M},\, \theta\ge0\}$ is the conical hull of $\tilde{s}$ in $\mathbb{M}$.
				\end{definition}
				\begin{definition}\label{abs4}
					Let $\tilde{s}\in\mathcal{S}$ be a local optimal solution for problem (\ref{abs1}). 
					It will be said that $\tilde{s}$ is a regular point if
					\begin{equation*}\label{abs5}
						G'(\tilde{s})[\mathcal{C}(\tilde{s})]
						=\mathbb{Y}.
					\end{equation*}
					\end{definition}
					\begin{theorem}(\cite[Theorem 6.3, p.330]{troltz}, \cite[Theorem 3.1]{zowe})\label{abs6}
						Let $\tilde{s}\in\mathcal{S}$ be a local  optimal solution for problem (\ref{abs1}).  Suppose that $J$ is Fr\'echet differentiable in $\tilde{s}$, and
						$G$ is continuous Fr\'echet-differentiable in $\tilde{s}$ . If $\tilde{s}$ is a regular point, then there exists  Lagrange multipliers for (\ref{abs1}) at $\tilde{s}$. 
					\end{theorem}
					
					\subsection{Application of the Lagrange multiplier theory}
					
					Now,  in order to reformulate the optimal control problem (\ref{func}) in the abstract setting  (\ref{abs1}), we introduce the Banach spaces
					\begin{equation*}\label{spaces_X_Y}
						\mathbb{X}:=\widehat{W}_2\times\widehat{X}_{2+}\times L^{2+}(Q_c),\ 
						\mathbb{Y}:=L^2((H^1)')\times L^{2+}(Q),
					\end{equation*}
					where 
					$$
					\widehat{W}_2 =\{u\in W_2\,:\, u(0)=0\},
					\quad \widehat{X}_{2+}=\{v\in X_{2+}\,:\, v(0)=0, \ \partial_n v\vert_{\partial\Omega}=0\}
					$$
					and the  operator $G=(G_1,G_2):\mathbb{X}\rightarrow\mathbb{Y}$, where 
					\begin{equation*}
						G_1:\mathbb{X}\rightarrow L^2((H^1)'),\quad G_2:\mathbb{X}\rightarrow L^{2+}(Q)
					\end{equation*}
					are defined at each point $s=(u,v,f)\in\mathbb{X}$ by
					\begin{equation*}\label{restriction}
						\left\{
						\begin{array}{l}
							\langle G_1(s),\varphi\rangle=\langle\partial_t u, \varphi\rangle_{L^2(H^1),L^2((H^1)')}
							+(\nabla u-\kappa  \, u\nabla v, \nabla \varphi )_{L^2}
							\\
							\noalign{\vspace{-2ex}}\\
							\qquad \qquad + (-r\, u + \mu\, u^2,\varphi )_{L^2}
							\quad \forall\, \varphi\in L^2(H^1)
							\vspace{0.1cm}\\
							G_2(s)=\partial_tv-\Delta v+v-u-f\,v\,1_{\Omega_c}
							\quad \hbox{in $L^{2+}(Q).$}
						\end{array}
						\right.
					\end{equation*}
					Thus, the optimal control problem (\ref{func}) is reformulated  as follows
					\begin{equation}\label{problem-1}
						\min_{s\in{\mathbb{M}}}J(s)\quad \mbox{ subject to }\quad {G}(s)={\bf 0},
					\end{equation}
					where
					\begin{equation}\label{Mset}
						\mathbb{M}:=(\hat{u},\hat{v},0)
						+\widehat{W}_2 \times \widehat{X}_{2+}\times \mathcal{F},
					\end{equation}
					with $(\hat{u},\hat{v})$ the global weak solution of (\ref{KS}) without control $(\hat{f}=0$) and $\mathcal{F}$ is defined in (\ref{F-def}).
					\begin{remark}\label{lagrag-problem}
						From Definition \ref{lagrangian},   the Lagragian associated to optimal control problem (\ref{problem-1}) is the
						functional $\mathcal{L}:\mathbb{X}\times L^{2}(H^1)\times L^{2-}(Q)\rightarrow\mathbb{R}$ given by
						$$
						\mathcal{L}(s,\lambda,\eta)=J(s)-\langle\lambda,G_1(s)\rangle_{L^2(H^1),L^2((H^1)')}-(\eta,G_2(s))_{L^{2-},L^{2+}}.
						$$
					\end{remark}
					
					The set ${\mathbb{M}}$ defined in (\ref{Mset}) is a closed convex subset of $\mathbb{X}$ 
					and the admissible  set of control problem (\ref{problem-1}) is
					\begin{equation}\label{admi-1}
						\mathcal{S}_{ad}=\{s=(u,v,f)\in{\mathbb{M}}\,:\, {G}(s)={\bf 0}\}.
					\end{equation}
					Concerning to the differentiability of the functional $J$ and the constraint operator ${G}$,  one has the following results.
					
					\begin{lemma}
						The functional  $J:\mathbb{X}\rightarrow\mathbb{R}$ is Fr\'echet differentiable and the  Fr\'echet
						derivative of $J$ in $\tilde{s}=(\tilde{u},\tilde{v},\tilde{f})\in\mathbb{X}$ in the direction
						$r=(U,V,F)\in\mathbb{X}$ is
						\begin{equation}\label{C7}
							\begin{split}
								J'(\tilde{s})[r]=\gamma_u\int_0^T\int_{\Omega} (\tilde u - u_d) U+\gamma_v\int_0^T\int_{\Omega}(\tilde{v}-v_d)V
								\\
								+\gamma_f\int_0^T\int_{\Omega_c} {\rm sgn}(\tilde{f})|\tilde{f}|^{1+}  F.
							\end{split}
						\end{equation}
					\end{lemma}
					\begin{lemma}
						The operator ${G}:\mathbb{X}\rightarrow\mathbb{Y}$  is continuous-Fr\'echet differentiable and the Fr\'echet derivative of  ${G}$ in 
						$\tilde{s}=(\tilde{u},\tilde{v},\tilde{f})\in\mathbb{X}$ in the direction $r=(U,V,F)\in\mathbb{X}$ is the linear operator
						${G}'(\tilde{s})[r]=(G_1'(\tilde{s})[r],G_2'(\tilde{s})[r])$ defined by 
						\begin{equation}\label{C8}
							\left\{
							\begin{array}{l}
								\langle G_1'(\tilde{s})[r], \varphi\rangle=
								\langle \partial_tU, \varphi\rangle
								+(\nabla U- \kappa \, U\nabla\tilde{v}-\kappa\,  \tilde{u}\nabla V, \nabla\varphi)
								\\ \qquad \qquad \qquad
								+ (-r\, U + 2\mu\, \tilde{u}\, U, \varphi), \quad \forall\, \varphi\in L^2(H^1)
								\vspace{0.1cm}\\
								G_2'(\tilde{s})[r]=\partial_tV-\Delta V+V-U-\tilde{f}\,V\,1_{\Omega_c}-F\tilde{v} \,1_{\Omega_c}.
							\end{array}
							\right.
						\end{equation}
					\end{lemma}
					
					\
					
					\begin{remark}\label{regular_point}
						From Definition \ref{abs4}  one has  that 
						$\tilde{s}=(\tilde{u},\tilde{v},\tilde{f})\in\mathcal{S}_{ad}$
						is a regular point  if for  any $(g_u,g_v)\in\mathbb{Y}$ there exists $r=(U,V,F)\in \widehat{W}_2\times \widehat{X}_{2+}\times\mathcal{C}(\tilde{f})$
						such that
						\begin{equation*}\label{regular}
							{G}'(\tilde{s})[r]=(g_u,g_v),
						\end{equation*}
						where $\mathcal{C}(\tilde{f}):=\{\theta(f-\tilde{f})\,:\, \theta\ge0,\, f\in\mathcal{F}\}$ is the conical hull of  $\tilde{f}$ in $\mathcal{F}$. 
					\end{remark}

					\subsection{The linearized problem \eqref{C8} is surjective}
					
					\begin{lemma}\label{regular-lemma}
						Let $\tilde{s}=(\tilde{u},\tilde{v},\tilde{f})\in\mathcal{S}_{ad}$  ($\mathcal{S}_{ad}$ defined in (\ref{admi-1})), then $\tilde{s}$ is a regular point.
					\end{lemma}
					
					\begin{proof}
						Fixed $(\tilde{u},\tilde{v},\tilde{f})\in\mathcal{S}_{ad}$, let  $(g_u,g_v)\in L^2((H^1)')\times L^{2+}(Q)$. Since $0\in\mathcal{C}(\tilde{f})=\{\theta(f-\tilde{f})\,:\, \theta\ge0,\, f\in\mathcal{F}\}$, it suffices to show  the existence of	$(U,V)\in \widehat{W}_2 \times \widehat{X}_{2+}$ solving the linear problem
						\begin{equation}\label{C9}
							\left\{
							\begin{array}{l}
								\langle \partial_tU, \varphi\rangle
								+(\nabla U- \kappa\, U\nabla\tilde{v}- \kappa\, \tilde{u}\nabla V, \nabla\varphi)
								\\
								\qquad + (-r\, U + 2\mu\, \tilde{u}\, U, \varphi) = \langle g_u, \varphi\rangle \quad\forall\, \varphi\in L^{2}(H^{1}),
								\vspace{0.1cm}\\
								\partial_tV-\Delta V+V-U-\tilde{f}\,V \,1_{\Omega_c}=g_v \quad \mbox{ in }L^{2+}(Q).
							\end{array}
							\right.
						\end{equation}
						For this, we will use  	the Leray-Schauder fixed-point Theorem, for the operator 
						\begin{equation}\label{operatorS}
							S:(\overline{U},\overline{V})\in L^{4-}(Q)\times L^{\infty}(Q)\mapsto(U,V)\in \widehat{W}_{2}\times \widehat{X}_{2+}
						\end{equation} 
						where $(U,V)$ is the solution of the decoupled	problem (first computing $V$ and after $U$)
						\begin{equation}\label{modlin}
							\left\{
							\begin{array}{l}
								\langle \partial_t U, \varphi\rangle
								+(\nabla U- \kappa \, U\nabla\tilde{v}- \kappa\,\tilde{u}\nabla V, \nabla\varphi)
								\\
								\qquad =  (r\, \overline U - 2\mu\, \tilde{u}\, \overline U, \varphi) 
								+ \langle g_u, \varphi\rangle \quad\forall\, \varphi\in L^{2}(H^{1}),
								\vspace{0.1cm} \\
								\partial_t V-\Delta V+V= \overline{U}+\tilde{f}\overline{V} 1_{\Omega_c}+g_v\quad\mbox{ in }Q .
							\end{array}
							\right.
						\end{equation}
						%
						%

						\
						
						{\sl Step 1 ($S$ is well-defined and bounded):} Prove  that operator $S$ defined in (\ref{operatorS}) maps bounded sets in $ L^{4-}(Q)\times L^{\infty}(Q)$  in bounded sets in $(U,V)\in {W}_{2}\times {X}_{2+}$. 
						
						\
						
						For this,  
						one first bound $V$ and later bound $U$. 
						Indeed,  since $(\overline{U}, \overline{V})\in L^{4-}(Q)\times L^{\infty}(Q)$ implies $f\, \overline{V} \in L^{2+}(Q)$, then by applying $L^{2+}$-regularity to the heat equation 
						(\ref{modlin})$_2$ (Theorem \ref{feireisl}), it is deduced that $V \in X_{2+}$ and 
						\begin{equation}\label{AA}
							\begin{array}{rcl}
								\Vert V \Vert_{X_{2+}} 
								& \le & C \, \left(
								\Vert 
								\overline{U}+\tilde{f}\overline{V} 1_{\Omega_c}\Vert_{L^{2+}(Q)}
								+\Vert g_v \Vert_{L^{2+}(Q)}
								\right)
								\vspace{0.1cm}\\
								& \le &
								C \, \left(
								\Vert 
								\overline{U}\Vert_{L^{2+}} 
								+
								\Vert \tilde{f}\Vert_{L^{2+}(Q_c)}
								\Vert \overline{V} \Vert_{L^{\infty}}
								+ \Vert g_v 
								\Vert_{L^{2+}(Q)}
								\right)
							\end{array}
						\end{equation}
						By taking $\varphi=U$ in (\ref{modlin})$_1$, we arrive at
						\begin{equation}\label{BB}
							\begin{array}{l}
								\frac{d}{dt}\Vert U \Vert^2_{L^2} 
								+ \Vert U \Vert^2_{H^1}
								\le  
						 C_1(1+ \Vert \nabla\tilde{v} \Vert_{L^{4}}^4) \Vert U \Vert^2_{L^2}
\vspace{0.1cm}\\ 
		\qquad \qquad
+C_2 \Big( 			\Vert \tilde{u} \Vert ^2_{L^{4}}  
								\Vert \nabla V\Vert^2_{L^{4}}
								 + (1+  \Vert \tilde{u} \Vert_{ L^{4}}^2)  \Vert \overline U\Vert_{L^{2}}^2
								+\Vert g_u \Vert_{(H^1)'}^2 \Big)
						\end{array}
						\end{equation}
						
						Finally, using $2D$ interpolation estimates, 
						$$
						\Vert \tilde{u} \Vert _{L^{4}(Q)}^2   \Vert \nabla V\Vert_{L^{4}(Q)}^2
						\le 
						\Vert \tilde{u} \Vert _{W_2}^2   \Vert V\Vert_{X_2}^2
						$$	
						Then, using (\ref{AA}), the Gronwall Lemma applied to (\ref{BB}) guarantees the bound   for $U$ in $W_2$.
						
						
						\
						
						{\sl Step 2 (compactness):} 
						By using that ${W}_{2}\times {X}_{2+}$ is compactly embedded in $L^{4-}(Q)\times L^{\infty}(Q)$, then operator $S$ is compact. 
						
						\
						
						{\sl Step 3 (continuity):} 
						%
						In particular, using Steps 1 and 2, it is not difficult to prove the continuity of $S$ from $L^{4-}(Q)\times L^{\infty}(Q)$ to itself. 
						
						\
						
						{\sl Step 4 (boundedness of possible fixed-points):} Now,  the aim is to show that the set
						of the possible fixed-points of $\alpha S$ with $\alpha\in [0,1]$  defined as 
						$S_\alpha:=\{(U,V)\in \widehat{W}_{2}\times \widehat{X}_{2+}\,:\, (U,V)=\alpha S(U,V)\mbox{ for some }\alpha\in[0,1]\}$ 
						is bounded in $L^{4-}(Q)\times L^{\infty}(Q)$ (with respect to $\alpha$).  
						Indeed, if $(U,V)\in S_\alpha,$ then $(U,V)\in \widehat{W}_{2}\times \widehat{X}_{2+}$ and solves the coupled 
						linear problem
						\begin{equation}\label{C9-1}
							\left\{
							\begin{array}{l}
								\langle \partial_tU, \varphi\rangle
								+(\nabla U- \kappa \,U\nabla\tilde{v}-\kappa\, \tilde{u}\nabla V, \nabla\varphi)
								\\
								\qquad =\alpha (r\, U - 2\mu\, \tilde{u}\, U, \varphi) + \alpha\langle g_u, \varphi\rangle 
								\quad\forall\, \varphi\in L^{2}(H^{1}),
								\vspace{0.1cm}
								\\
								\partial_tV-\Delta V+V=\alpha U+\alpha\tilde{f}V 1_{\Omega_c}+\alpha g_v\ \mbox{ in }Q .
							\end{array}
							\right.
						\end{equation}
						Then, taking $\varphi=U$ in (\ref{C9-1})$_1$, one obtains (see \eqref{BB}): 
						\begin{equation}\label{lin1}
							\begin{array}{l}\displaystyle
								\frac{d}{dt} \Vert U \Vert^2 + \Vert \nabla U \Vert^2
								+2 \alpha  \,  \mu  \, \int_\Omega  \tilde{u}\,  U ^2
								\\
								\qquad \le C  \left(
								\alpha +\Vert \nabla \tilde{v} \Vert_{L^4}^4
								\right) \, \Vert U \Vert^2 
								+ C \, \Vert \tilde{u} \Vert_{L^{4}}^4 \Vert \nabla V \Vert^2
								+ \alpha^2 \, \Vert g_u \Vert_{(H^1)'}^2 .
							\end{array}
						\end{equation}

						Now, testing (\ref{C9-1})$_2$ by $V-\Delta V \in L^{2+}(Q)$, it holds:
						\begin{equation}\label{lin2}
							\frac{d}{dt} \Vert V \Vert_{H^1}^2 + \Vert V \Vert_{H^2}^2 \le 
							C \, \alpha^2 \Vert f \Vert_{L^{2+}}^2 \, \Vert V \Vert_{H ^1}^2
							+ \alpha^2 \, \left(
							\Vert g_v \Vert^2 + \Vert U \Vert^2 
							\right).
						\end{equation}
						From (\ref{lin1}) and (\ref{lin2}) and using that $\alpha\le 1$:
						$$
						\begin{array}{l}
							\displaystyle \frac{d}{dt} \left(
							\Vert U \Vert^2 +  \Vert V \Vert_{H^1}^2 \right)
							+  \| U \Vert_{H^1}^2 +
							\Vert V \Vert_{H^2}^2
							\vspace{0.1cm}\\
							\le 
							C \, \left(
							1+\Vert \nabla \tilde{v} \Vert_{L^4}^4
							\right) \, \Vert U \Vert^2 
							+ C \,  \Big( \Vert f \Vert_{L^{2+}}^2 
							+ \Vert \tilde{u} \Vert_{L^{4}}^4
							\Big)
							\Vert V \Vert_{H^1}^2
							\vspace{0.1cm}\\
							+   
							C \, \Big(
							\Vert g_u \Vert_{(H^1)'}^2 
							+ \Vert g_v \Vert^2
							\Big).
						\end{array}
						$$
						Using that $U(0)=V(0)=0$ and $\Vert g_u \Vert_{L^2((H^1)')}$, $\Vert g_v \Vert_{L^2(L^2)}$, $\Vert f \Vert_{L^{2+}}$, $\Vert \tilde{u} \Vert_{L^{4}(Q)}$ and
						$\Vert\nabla \tilde{v} \Vert_{L^4(Q)}$ are constant finite values, then the Gronwall Lemma implies that
						\begin{equation}\label{C19-b}
							\Vert (U,V) \Vert_{W_2\times X_2} \le C.
						\end{equation}
						
						Finally, by applying $L^{2+}$-regularity provided by Theorem \ref{feireisl} related to the parabolic-Neumann problem, one has 
						$$
						\Vert V \Vert_{X_{2+}} \le C.
						$$
						\
						
						{\sl Step 5:} Conclusion: By applying Leray-Schauder fixed-point theorem (Theorem \ref{LSFP}), one has the existence of $(U,V)\in W_2\times X_{2+}$ a solution of  problem \eqref{C9}. The uniqueness of solution is directly deduced from the linearity of  problem \eqref{C9}.
					\end{proof}

					\subsection{Existence of Lagrange multipliers}

					Now, the existence of Lagrange multiplier for problem (\ref{func}) associated to any local optimal solution $\tilde{s}=(\tilde{u},\tilde{v},\tilde{f})\in\mathcal{S}_{ad}$ will be shown.
					\begin{theorem}\label{lagrange}
						Let $\tilde{s}=(\tilde{u},\tilde{v},\tilde{f})\in\mathcal{S}_{ad}$ be a local optimal solution for the control problem (\ref{func}). Then, there exists a Lagrange multiplier 
						$\xi=(\lambda,\eta)\in L^{2}(H^1)\times L^{2-}(Q)$ such that for all 
						$(U,V,F)\in \widehat{W}_2 \times \widehat{X}_{2+}\times\mathcal{C}(\tilde{f})$
						\begin{eqnarray}
							&&\gamma_u\int_0^T\int_{\Omega} (\tilde{u}-u_d) U+\gamma_v\int_0^T\int_{\Omega}(\tilde{v}-v_d)V
							+\gamma_f\int_0^T\int_{\Omega_c}  {\rm sgn}(\tilde{f})|\tilde{f}|^{1+} F
							\nonumber\\
							&&
							-\int_0^T \langle \partial_t U, \lambda\rangle -\int_0^T\int_\Omega 
							(\nabla U- \kappa \, U\nabla\tilde{v}- \kappa\,\tilde{u}\nabla V,\nabla\lambda)
							+ (-r U + 2 \mu \tilde{u} U , \lambda)
							\nonumber\\
							&&-\int_0^T\int_\Omega\bigg(\partial_tV-\Delta V+V-U-\tilde{f}V 1_{\Omega_c}\bigg)\eta+\int_0^T\int_{\Omega_c}F\tilde{v} 1_{\Omega_c} \eta\ge0.\label{M1}
						\end{eqnarray}
					\end{theorem}
					\begin{proof}
						From Lemma \ref{regular-lemma}, $\tilde{s}\in\mathcal{S}_{ad}$ is a regular point, then from Theorem \ref{abs6} 
						there exists a Lagrange multiplier $\xi=(\lambda,\eta)\in L^{2}(H^1)\times L^{2-}(Q)$ such that
						by (\ref{abs3})$_2$ and Remark \ref{lagrag-problem} one must satisfy
						\begin{equation}\label{M1-1}
							\mathcal{L}'_s(s,\lambda,\eta)[r]=J'(\tilde{s})[r]-\langle\lambda, G_1'(\tilde{s})[r]\rangle_{L^{2}(H^1), L^{2}((H^1)'}-( \eta, G_2'(\tilde{s})[r])_{L^{2}}\ge 0,
						\end{equation}
						for all $r=(U,V,F)\in\widehat{W}_2\times \widehat{X}_{2+} \times\mathcal{C}(\tilde{f}).$ 
						The proof follows from (\ref{C7}), (\ref{C8}), and (\ref{M1-1}).
					\end{proof}
					
					\
					
					From Theorem \ref{lagrange},  an optimality system for problem (\ref{func}) can be derived. 
					\begin{corollary}
						Let $\tilde{s}=(\tilde{u},\tilde{v},\tilde{f})\in\mathcal{S}_{ad}$ be a local optimal solution for the control problem (\ref{func}). Then any Lagrange multiplier
						$(\lambda,\eta)\in L^{2}(H^1)\times L^{2-}(Q)$, provided by Theorem \ref{lagrange}, satisfies the system
						\begin{eqnarray}
							&&\int_0^T\langle \partial_tU, \lambda\rangle
							+\int_0^T\int_\Omega (\nabla U- \kappa \,U\nabla\tilde{v})\cdot \nabla \lambda
							+ (-r U + 2 \mu \tilde{u} U , \lambda)
							-\int_0^T\int_\Omega U\eta
							\nonumber\\
							&&\hspace{0.7cm}=\gamma_u\int_0^T\int_{\Omega} (\tilde{u}-u_d) U,\qquad  \forall \, U\in\widehat{W}_2\label{M2},\\
							&&\int_0^T\int_\Omega\bigg(\partial_t V-\Delta V+V\bigg)\eta
							-\int_0^T\int_{\Omega_c}\tilde{f}V\eta
							+ \kappa\int_0^T\int_\Omega \tilde{u}\nabla V \cdot \nabla \lambda
							\nonumber\\
							&&\hspace{0.7cm}=\gamma_v\int_0^T\int_{\Omega}(\tilde{v}-v_d)V,\qquad  \forall \,V\in\widehat{X}_{2+}\label{M3},
						\end{eqnarray}
						and the optimality condition
						\begin{equation}\label{M5}
							\int_0^T\int_{\Omega_c}(\gamma_f {\rm sgn}(\tilde{f}) |\tilde{f}|^{1+}+\tilde{v}\eta)(f-\tilde{f})\ge 0,\qquad \forall f\in\mathcal{F}.
						\end{equation}
					\end{corollary}
					\begin{proof}
						From (\ref{M1}), taking  $(V,F)=(0,0)$, and using that $\widehat{W}_2$ is a vectorial space,  (\ref{M2}) holds. Similarly, taking $(U,F)=(0,0)$ in (\ref{M1}), and taking into account
						that $\widehat{X}_{2+}$ is a vectorial space,  (\ref{M3}) is deduced.
						Finally, taking $(U,V)=(0,0)$ in (\ref{M1})  it holds
						$$
						\gamma_f\int_0^T\int_{\Omega_c} {\rm sgn}(\tilde{f}) |\tilde{f}|^{1+} F+\int_0^T\int_{\Omega_c}\tilde{v}\, \eta\, F\ge0,\quad \forall \, F\in\mathcal{C}(\tilde{f}).
						$$
						Thus, choosing $F=\theta(f-\tilde{f})\in \mathcal{C}(\tilde{f})$ for all $f\in\mathcal{F}$ and $\theta\ge0$, 
						(\ref{M5}) is deduced. 
					\end{proof}
					\begin{remark}\label{very-weak}
						A pair $(\lambda,\eta)\in L^{2}(H^1) \times L^{2-}(Q)$ satisfying (\ref{M2})-(\ref{M3}) corresponds to the concept of very weak solution (at least for the $\eta$-variable)
						of the linear problem
						\begin{equation}\label{M4}
							\left\{
							\begin{array}{rcl}
								-\partial_t\lambda-\Delta\lambda
								-\kappa\,\nabla\lambda\cdot\nabla\tilde{v}
								-\eta
								-r \lambda + 2 \mu \tilde{u} \lambda 
								&=&\gamma_u (\tilde{u}-u_d) \quad \mbox{ in }Q,\\
								-\partial_t\eta-\Delta\eta+\eta
								-\kappa\,\nabla\cdot(\tilde{u}\nabla\lambda)-\tilde{f}\,\eta\, 1_{\Omega_c}&=&\gamma_v(\tilde{v}-v_d)\quad \mbox{ in }Q,\\
								\lambda(T)=0 ,\ \ \eta(T)&=&0\quad \mbox{ in }\Omega,\\
								\dfrac{\partial\lambda}{\partial\n}=0,\ \dfrac{\partial\eta}{\partial\n}&=&0\quad \mbox{ on }(0,T)\times\partial\Omega.
							\end{array}
							\right.
						\end{equation}
					\end{remark}
					
					\subsection{Regularity of Lagrange multipliers}
					
					%
					%
					%
					%

					\begin{theorem}\label{regularity_lagrange}
						Let $\tilde{s}=(\tilde{u},\tilde{v},\tilde{f})\in\mathcal{S}_{ad}$ be a local optimal solution for the problem (\ref{func}). Then
						the problem (\ref{M4}) has a unique solution $(\lambda,\eta)$ such that
						\begin{equation}
							(\lambda,\eta)\in X_{2}\times W_{2}
							\label{reg1}
						\end{equation}
					\end{theorem}
					\begin{proof}
						Let $s=T-t$, with $t\in(0,T)$
						and $\tilde{\lambda}(s)=\lambda(t)$, $\tilde{\eta}(s)=\eta(t)$. Then, system (\ref{M4}) is equivalent to 
						\begin{equation}\label{R1}
							\left\{
							\begin{array}{rcl}
								\partial_s\tilde{\lambda}-\Delta\tilde{\lambda}
								-\kappa\,\nabla\tilde{\lambda}\cdot\nabla\tilde{v}-\tilde{\eta}
								-r \tilde\lambda + 2 \mu \tilde{u} \tilde\lambda 
								&=&\gamma_u (\tilde{u} - u_d)\quad \mbox{ in }Q,\\
								\partial_s\tilde{\eta}-\Delta\tilde{\eta}+\tilde{\eta}
								-\kappa\,\nabla\cdot(\tilde{u}\nabla\tilde{\lambda}) -\tilde{f}\,\tilde{\eta}\, 1_{\Omega_c}&=&\gamma_v(\tilde{v}-v_d)\quad \mbox{ in }Q,\\
								\tilde{\lambda}(0)=0,\ \tilde{\eta}(0)&=&0\quad \mbox{ in }\Omega,\\
								\dfrac{\partial\tilde{\lambda}}{\partial\n}=0,\ \dfrac{\partial\tilde{\eta}}{\partial\n}&=&0\quad \mbox{ on }(0,T)\times\partial\Omega.
							\end{array}
							\right.
						\end{equation}
						
						\
						
						In order to prove the existence  of a solution for (\ref{R1}), 
						the Leray-Schauder  fixed-point Theorem can be applied  as before, now over the operator
						\begin{equation}\label{operator-That}
							\widehat{T}:(\bar{\lambda},\bar{\eta})\in L^{\infty-}\times L^{4-}\mapsto(\lambda,\eta)\in X_{2}\times W_{2}
						\end{equation} 
						where 
						$(\lambda,\eta)=\widehat{T}(\bar{\lambda},\bar{\eta})$ solves the decoupled problem (first computing $\lambda$ and after $\mu$):
						\begin{equation}\label{R1-bis}
							\left\{
							\begin{array}{rcl}
								\partial_s {\lambda}-\Delta {\lambda}
								-\kappa\,\nabla{\lambda}\cdot\nabla\tilde{v}
								&=& \bar{\eta} 
								+r \bar\lambda - 2 \mu \tilde{u} \bar\lambda 
								+ \gamma_u (\tilde{u}-u_d)\quad \mbox{ in }Q,\\
								\partial_s {\eta}-\Delta  \eta
								+{\eta}
								- \kappa\,\nabla\cdot(\tilde{u}\nabla{\lambda})
								&=&
								\tilde{f}\, \bar{\eta}\, 1_{\Omega_c}
								+ \gamma_v(\tilde{v}-v_d)
								\quad \mbox{ in }Q,\\
								{\lambda}(0)=0,\ {\eta}(0)&=&0\quad \mbox{ in }\Omega,\\
								\dfrac{\partial{\lambda}}{\partial\n}=0,\ \dfrac{\partial{\eta}}{\partial\n}&=&0\quad \mbox{ on }(0,T)\times\partial\Omega,
							\end{array}
							\right.
						\end{equation}
						The proof follows the same lines and it will be omitted. Indeed, the key point is to show that the set of possible fixed-points 
						$$
						\widehat T_{\alpha}:=\{(\lambda,\eta)\in X_{2} \times W_{2}\,:\, (\lambda,\eta)=
						\alpha \widehat{T}(\lambda,\eta) \mbox{ for some }{\alpha}\in[0,1]\}
						$$ 
						is bounded in $X_{2} \times W_{2}$ (with respect to ${\alpha}$).  In fact, if $(\lambda,\eta)\in \widehat T_{\alpha},$ then $(\lambda,\eta)\in {X}_{2}\times {W}_{2}$ and solves the coupled linear problem:
						\begin{equation}\label{R1-bis-2}
							\left\{
							\begin{array}{rcl}
								\partial_s {\lambda}-\Delta {\lambda}
								+{\kappa} \, \nabla{\lambda}\cdot\nabla\tilde{v}
								-\alpha \,r \lambda + 2 \alpha \, \mu \tilde{u} \lambda 
								- {\alpha} \, {\eta}
								&=&
								{\alpha} \, \gamma_u (\tilde{u}- u_d)\quad \mbox{ in }Q,
								\\
								\partial_s {\eta}-\Delta  \eta
								+{\eta}-  \tilde{f}\, {\eta}\, 1_{\Omega_c}
								-\kappa\,\nabla\cdot(\tilde{u}\nabla{\lambda})
								&=&
								{\alpha} \, \gamma_v(\tilde{v}-v_d)\quad \mbox{ in }Q,\\
								{\lambda}(0)=0,\ {\eta}(0)&=&0\quad \mbox{ in }\Omega,\\
								\dfrac{\partial{\lambda}}{\partial\n}=0,\ \dfrac{\partial{\eta}}{\partial\n}&=&0\quad \mbox{ on }(0,T)\times\partial\Omega.
							\end{array}
							\right.
						\end{equation}
						Now, by taking $\lambda - \Delta \lambda \in L^2(Q)$ as test function in (\ref{R1-bis-2})$_1$ and  $\eta \in L^2(H^1)$ as test function in (\ref{R1-bis-2})$_2$, then the following bound is obtained via the Gronwall Lemma 
						$$
						\Vert (\lambda,\eta) \Vert_{X_2\times W_2} 
						\le 
						C( \Vert \tilde{u} \Vert_{W_2},
						\Vert \tilde{v} \Vert_{X_2},
						\Vert \tilde{f} \Vert_{L^{2+}(Q_c)}, 
						\Vert u_d \Vert_{L^2(Q)},
						\Vert v_d \Vert_{L^2(Q)} ).
						$$

						Therefore, by applying Leray-Schauder fixed-point theorem, 
						the existence of a solution of  problem \eqref{M4}, 
						$(\lambda,\eta)\in X_{2}\times W_{2}$, is obtained. 
						The uniqueness of solution is directly deduced from the linearity of  problem \eqref{M4}.
					\end{proof}
					
					\
					
					In the following result,  more regularity and uniqueness of the  Lagrange multiplier $(\lambda,\eta)$ than provided by Theorem \ref{lagrange} will be obtained via uniqueness of the problem \eqref{M4}.
					\begin{theorem}\label{regularity-multiplier}
						Let $\tilde{s}=(\tilde{u},\tilde{v},\tilde{f})\in\mathcal{S}_{ad}$ be a local optimal solution for the control problem (\ref{func}). Then the Lagrange
						multiplier, provided by Theorem \ref{lagrange}, is unique and satisfies $(\lambda,\eta)\in X_{2}\times W_{2}$.
					\end{theorem}
					\begin{proof}
						Let $(\lambda,\eta)\in L^{2}(H^1)\times L^{2-}(Q)$ be a Lagrange multiplier given in Theorem \ref{lagrange}, which is a very weak solution of problem (\ref{M4}). 
						In particular, $(\lambda,\eta)$ satisfies (\ref{M2})-(\ref{M3}). On the other hand, from Theorem \ref{regularity_lagrange}, 
						system (\ref{M4}) has a unique solution $(\overline{\lambda},\overline{\eta})\in X_{2}\times W_{2}$. 
						Then, it suffices to identify
						$(\lambda,\eta)$ with $(\overline{\lambda},\overline{\eta})$. 
						
						\
						
						With this objective,  for any  $(U,V)\in \widehat{W}_2 \times \widehat{X}_{2+}$, we 
						writte (\ref{M4}) for $(\overline{\lambda},\overline{\eta})$ (instead of $(\lambda,\eta)$), 
						testing the first equation by $U$, and the second one by $V$, 
						and  integrating by parts in $\Omega$, it is obtained
						\begin{equation}\label{RL2}
							\int_0^T \langle \partial_tU, \overline{\lambda} \rangle 
							+ \int_0^T \int_\Omega (\nabla U
							-\kappa \,U\nabla\tilde{v}) \cdot \nabla\overline{\lambda}
							+ (-r U + 2 \mu \tilde{u} U ) \overline\lambda
							- U\overline{\eta}
							=\gamma_u\int_0^T\int_{\Omega} (\tilde{u}-u_d) U,
						\end{equation}
						\begin{equation}\label{RL3}
							\int_0^T\int_\Omega\bigg(\partial_tV-\Delta V+V-\tilde{f}V 1_{\Omega_c}\bigg)\overline{\eta}
							+ \kappa\, \tilde{u}\nabla V\cdot \nabla\overline{\lambda}
							=\gamma_v\int_0^T\int_{\Omega}(\tilde{v}-v_d)V.
						\end{equation}
						Making the difference between (\ref{M2}) for $(\lambda,\eta)$ and (\ref{RL2}) for $(\overline{\lambda},\overline{\eta})$, and between (\ref{M3}) and (\ref{RL3}), and then adding the respective 
						equations, since the right-hand side terms vanish,   it can be deduced
						\begin{eqnarray*}
							&& \int_0^T \langle \partial_tU, \lambda-\overline{\lambda} \rangle_{(H^1)'}
							+ \int_0^T \int_\Omega (\nabla U
							-\kappa \,U\nabla\tilde{v}- \kappa\, \tilde{u}\nabla{V}) \cdot \nabla(\lambda - \overline{\lambda})
							\nonumber\\
							&& + \int_0^T \int_\Omega (-r U + 2 \mu \tilde{u} U , \lambda-\overline{\lambda})
							\nonumber\\
							&&
							+\int_0^T\int_\Omega\bigg(\partial_tV-\Delta V+V-U-\tilde{f}V 1_{\Omega_c}\bigg)(\eta-\overline{\eta})=0.
						\end{eqnarray*}
						Then, if  $(U,V)\in \widehat{W}_2 \times \widehat{X}_{2+}$ is 
						the unique solution  of linear system 
						(\ref{C9}) associated to any  $(g_u,g_v)\in L^2((H^1)')\times L^{2+}(Q)$ (given by Lemma \ref{regular-lemma}), we arrive at 
						\begin{equation}\label{RL4}
							\int_0^T\langle g_u, \lambda-\overline{\lambda} \rangle_{(H^1)'}+\int_0^T\int_\Omega g_v(\eta-\overline{\eta})=0.
						\end{equation}
						%
						%
						%
						By density arguments, it is easy to deduce that $\lambda-\overline{\lambda} =0$ and $\eta-\overline{\eta}=0$, 
						which implies that $(\lambda,\eta)=(\overline{\lambda},\overline{\eta})$. 
						As a consequence of the regularity of $(\overline{\lambda},\overline{\eta})$,  it holds that
						$(\lambda,\eta)\in X_{2}\times W_{2}$.
					\end{proof}
					All previous arguments of Section \ref{Sec:LM} prove Theorem \ref{Th:Optimality-System}.
					
					\addcontentsline{toc}{section}{Appendix: Existence of Strong Solutions of Problem (\ref{regg-1})}



\begin{thebibliography}{99}
						
						
						
						\bibitem{Abergel}  Abergel, F., Casas, E.: Some optimal control problems of multistate equations appearing in fluid mechanics. RAIRO Modél.
						Math. Anal. Numér. Vol 27,  223–247 (1993).
						
						\bibitem{adams} Adams, R.: Sobolev spaces.
						Academic Press, New York (1975).
						
						\bibitem{Amann} Amann, H.: Nonhomogeneous linear and quasilinear elliptic and parabolic boundary value problems, in: Function Spaces, Differential Operators and Nonlinear Analysis. H. Triebel and H.J. Schmeiser (eds.), Teubner-texte Math. 133, Teubner, Stuttgart, 1993, 9-126.
						
						\bibitem{BBTW} Bellomo, N.; Bellouquid, A.; Tao, Y.; Winkler, M. 
						{\sl Toward a mathematical theory of Keller-Segel models of pattern formation in biological tissues}. 
						Math. Models Methods Appl. Sci. 25 (2015), no. 9, 1663--1763.
						
						
						\bibitem{Berselli} Berselli L.C.; Fan, J.:  Logarithmic and improved regularity criteria for the 3D nematic liquid crystals models, Boussinesq system, and MHD equations in a bounded domain. Commun. Pure Appl. Anal. 14 (2015), no. 2, 637-655.
						
						\bibitem{brezis} Br\'ezis, H.: Functional analysis. Sobolev spaces and partial differential equations. Springer, New York (2010).
						%
						%
						
						\bibitem{casas} Casas, E.:  An optimal control problem governed by the evolution Navier-Stokes equations. In Optimal control of viscous flows,
						Frontiers in applied mathematics, edited by S.S. Sritharan. SIAM, Philadelphia (1998).
						
						\bibitem{chaves-guerrero-1} Chaves-Silva, F.W., Guerrero, S.: A uniform controllability for the Keller-Segel system.
						Asymptot. Anal. Vol 92, no. 3-4, 313-338 (2015).
						
						\bibitem{chaves-guerrero-2} Chaves-Silva, F.W., Guerrero, S.: A controllability result for a chemotaxis-fluid model.
						J. Diff. Equations Vol. 262, no. 9, 4863-4905 (2017).
						
						\bibitem{cieslak} Cieslak, T., Lauren\c cot, P., Morales-Rodrigo, C.: Global existence and convergence to steady states in a chemorepulsion system.
						Parabolic and Navier-Stokes equations. Part 1, 105-117, Banach Center Publ., 81, Part 1, Polish Acad. Sci. Inst. Math., Warsaw, 2008.
						
						\bibitem{dearaujo} De Araujo, A.L.A., Magalh\~aes, P.M.D.: Existence of solutions and optimal control for a model of tissue invasion
						by solid tumours. J. Math. Anal. Appl. Vol. 421, 842-877 (2015).
						
						\bibitem{Rodriguez-Ferreira-Roa} Duarte-Rodríguez, A., Ferreira, L. C. F., Villamizar-Roa, E. J.: Global existence for an attraction-repulsion chemotaxis-fluid
						model with logistic source. Discrete Contin. Dyn. Syst. Ser. Vol. 24, 423-447 (2019).  
						
						\bibitem{feireisl} Feireisl, E., Novotn\'y, A.: Singular limits in thermodynamics of viscous fluids. Advances in Mathematical Fluid Mechanics. Birkh\"auser Verlag, Basel (2009).
						
						\bibitem{fister-mccarthy} Fister, K.R., Mccarthy, C.M.: Optimal control of a chemotaxis system. Quart. Appl. Math. Vol. 61, no. 2, 193-211 (2003).
						
						\bibitem{friedman} Friedman, A.: Partial differential equations. Holt, Rinehart and Winston Inc. (1969).
						
						\bibitem{guillen-mallea-rodriguez} Guill\'en-Gonz\'alez, F., Mallea-Zepeda, E., Rodr\'iguez-Bellido, M.A.: Optimal bilinear control problem related to a chemo-repulsion system
						in 2D domains. ESAIM Control Optim. Calc. Var. 26 (2020), Paper No. 29, 21 pp. 
						DOI:10.1051/cocv/2019012. 
						
						\bibitem{guillen-mallea-villamizar} Guillén-González, F.; Mallea-Zepeda, E.; Villamizar-Roa, E.J:  On a Bi-dimensional chemo-repulsion model with nonlinear production and a related optimal control problem. Acta Applicandae Mathematicae.
						Vol.170, No.1, 963 - 979 (2020) 
						
						\bibitem{guillen-mallea-rodriguez3}
						 Guillén-González, F.; Mallea-Zepeda, E.; Rodríguez-Bellido, M.A. A Regularity Criterion for a 3D Chemo-Repulsion System and Its Application to a Bilinear Optimal Control Problem. SIAM J. Control and Optimization, 58, no. 3 (2020), 1457-1490. 

						
						\bibitem{karl-wachsmuth} Karl, V., Wachsmuth, D.: An augmented Lagrange method for elliptic  state constrained optimal control problems.
						Comp. Optim. Appl. Vol. 69, 857-880 (2018).
						
						\bibitem{keller-segel} Keller, E.F., Segel, L.A.: Initiation of slime mold aggregation viewed as an instability.
						J. Theo. Biol. Vol. 26, 399-415 (1970).
						
						\bibitem{Lankeit2016} Lankeit, J.: Long-term behaviour in a chemotaxis-fluid system with logistic source. Math. Models Methods Appl. Sci. Vol. 26, 2071–2109 (2016)
						
						\bibitem{lions} Lions, J.L.; Quelques m\'etodes de r\'esolution des probl\`emes aux limites non lin\'eares. Dunod, Paris, 1969.
						
						\bibitem{lions-magenes} Lions, J.L., Magenes, E.: Probl\`emes aux limites non homog\`enes et applications, Vol. 1. Travaux et recherches math\'ematiques, No. 17 Dunod, Paris 1968.
						
						\bibitem{Liu-Lorz-2011} Liu J.-G., Lorz, A.:  A coupled chemotaxis-fluid model: Global existence, Ann. Inst.H. Poincaré Anal. Non Linéaire Vol. 28, 643-652 (2011).   
						
						\bibitem{Lopez-Roa} López-Ríos, J., Villamizar-Roa, E. J.: An optimal control problem related to a 3D-Chemotaxis-Navier-Stokes model, ESAIM Control optim. Calc. Var. 27, 37pp (2021). 
						
						\bibitem{Zepeda-Torres-Roa} Mallea-Zepeda, E.,  Ortega-Torres, E., Villamizar-Roa, E. J.: A boundary control problem for micropolar fluids. J. Optim.
						Theory Appl. Vol 169, 349–369 (2016). 
						
						\bibitem{rodriguez-rueda-villamizar} Rodr\'iguez-Bellido, M.A., Rueda G\'omez, D.A., Villamizar-Roa, E.J.: On a distributed control problem for a coupled chemotaxis-fluid model.
						Discrete Contin. Dyn. Syst. B. Vol. 23, no. 2, 557-571 (2018).
						
						\bibitem{ryu-yagi} Ryu, S.-U., Yagi, A.: Optimal control of Keller-Segel equations.
						J. Math. Anal. Appl. Vol. 256, no. 1, 45-66 (2001).
						
						\bibitem{ryu} Ryu, S.-U.: Boundary control of chemotaxis reaction diffusion system.
						Honam Math. J. Vol. 30, no. 3, 469-478 (2008).
						
						\bibitem{simon} Simon, J.; Compact sets in the space $L^p(0,T;B)$. Ann. Mat. Pura Appl. Vol. 146, 65-96, 1987.
						
						\bibitem{tao} Tao, Y.: Global dynamics in a higher-dimensional repulsion chemotaxis model with nonlinear sensitivity. Discrete Contin. Dyn. Syst. B. Vol. 18, no. 10, 2705-2722 (2013).
						
						\bibitem{troltz} Tr\"{o}ltzsch, F.: Optimal control of partial differential equations. Theory, methods and applications. AMS Providence, Rhode Island (2010).
						
						\bibitem{winkler2012} Winkler, M.: Global large-data solutions in a chemotaxis-Navier-Stokes system modeling cellular swimming in fluid drops. Comm. Partial Differential Equations. Vol 37, 319-351 (2012).
						
						\bibitem{winklerJFA2019} Winkler, M.: A three-dimensional Keller-Segel-Navier-Stokes system with logistic source: global weak solutions and asymptotic stabilization.  J. Funct. Anal. Vol. 276, 1339–1401 (2019).
						
						
						\bibitem{zowe} Zowe, J., Kurcyusz, S.: Regularity and stability for the mathematical programming problem in Banach spaces.
						Appl. Math. Optim. Vol. 5, 49-62 (1979).
						
					\end{thebibliography}
				\end{document}